%% file: main.tex
\documentclass[11pt]{article}

\usepackage[a4paper,margin=1in]{geometry}
\usepackage{amsmath,amssymb,amsthm,mathtools}
\usepackage{enumitem}
\usepackage{microtype}
\usepackage[unicode]{hyperref}
\usepackage[nameinlink,capitalize,noabbrev]{cleveref}
\usepackage{doi}
\usepackage{authblk}
\usepackage{tikz}
\usepackage{tikz-cd}
\input{macros}

\title{\textbf{Barr-Exactness and Congruence-Based Homological Algebra in Ternary $\Gamma$-Semimodules}}

\author[1,2]{Chandrasekhar Gokavarapu}
\author[3,4]{Dr D Madhusudhana Rao}

\affil[1]{Lecturer in Mathematics, Government College(A), Rajahmundry, A.P.,India}
\affil[2]{Research Scholar, Dept of  Mathematics, Acharya Nagarjuna University, A.P., India}

\affil[3]{Lecturer in Mathematics, Government College for Women(A), Guntur, A.P.,India}
\affil[4]{Research Supervisor, Dept of Mathematics, Acharya Nagarjuna University,A.P., India}

\begin{document}
\maketitle

\begin{abstract}
Let $T$ be a ternary $\Gamma$--semiring and let $\mathsf{SMod}_{T,\Gamma}$ be the category of
ternary $\Gamma$--semimodules with the intrinsic five--ary action
$T\times \Gamma \times M \times \Gamma \times T\to M$.This paper develops a congruence--first calculus in $\mathsf{SMod}_{T,\Gamma}$.Kernels and quotients are taken in the sense of congruences, not in the sense of cosets of subobjects, because coset quotients generally fail without subtraction. We prove that every morphism admits a canonical factorization through its kernel congruence, establish congruence versions of the three isomorphism theorems, and describe coequalizers as quotients by generated congruences. Exactness is formulated in the regular/Barr-exact sense: kernels are kernel pairs and exactness is expressed by equality of kernels with regular images. Finally, we record a projective generation criterion and a concrete failure of a classical diagram lemma in this Barr-exact but non-abelian setting.

\end{abstract}

% Keywords and MSC (edit wording to match the journal's style file)


\begin{classification}
Primary 16Y60; Secondary 16Y80, 18E08, 18G50, 08A30
\end{classification}

\input{sections/01-introduction}

\input{sections/02-preliminaries}
\input{sections/03-ternary-gamma-semiring}
\input{sections/04-ternary-gamma-semimodules}
\input{sections/05-congruences-quotients}
\input{sections/06-isomorphism-theorems}
\input{sections/07-category-structure}
\input{sections/08-exactness-framework}
\input{sections/09-projectives-flatness}
\input{sections/10-derived-constructions}
\input{sections/11-examples-counterexamples}
\input{sections/12-spectrum-optional}
\input{sections/13-conclusion}

% Bibliography
\bibliographystyle{amsplain}
\bibliography{refs}

\end{document}

%% file: macros.tex
\theoremstyle{plain}
\newtheorem{theorem}{Theorem}[section]
\newtheorem{proposition}[theorem]{Proposition}
\newtheorem{lemma}[theorem]{Lemma}
\newtheorem{corollary}[theorem]{Corollary}

\theoremstyle{definition}
\newtheorem{definition}[theorem]{Definition}
\newtheorem{example}[theorem]{Example}
\newtheorem{remark}[theorem]{Remark}
\newtheorem{openproblem}[theorem]{Open Problem}
\newtheorem{hypothesis}[theorem]{Hypothesis}

\DeclareMathOperator{\Hom}{Hom}

\DeclareMathOperator{\Spec}{Spec}
\DeclareMathOperator{\fp}{fp}
\DeclareMathOperator{\Iso}{Iso}

\makeatletter
\@ifundefined{keywords}{%
  \newenvironment{keywords}{%
    \par\addvspace{6pt}%
    \begingroup\small\noindent\textbf{Key words and phrases.}\ }%
  {\par\endgroup\addvspace{6pt}}%
}{}

\@ifundefined{classification}{%
  \newenvironment{classification}{%
    \par\addvspace{0pt}%
    \begingroup\small\noindent\textbf{2020 Mathematics Subject Classification.}\ }%
  {\par\endgroup\addvspace{6pt}}%
}{}
\makeatother

%% file: sections/01-introduction.tex
\section{Introduction}

Work over semirings replaces subtraction by order, idempotency, or positivity constraints.
This replacement is structural, not cosmetic. It changes quotients, kernels, and exactness.
Recent module theory over semirings makes this explicit at the level of homological devices
and stability properties \cite{BorgerJun2025,AbuhlailNoegraha2021,Abuhlail2022}.
In parallel, tropical and ``positive'' geometries require congruences rather than ideals
as basic quotient data \cite{JunRay2020,Song2024,Ito2025,Sen2025,JooMincheva2025}.
For background on tropical combinatorics and tropical geometry motivating congruence methods,
see \cite{Joswig2021,AllamigeonGaubertSkomra2020,FriedenbergMincheva2022,Manjunath2024}.

\subsection*{Scope and the main choice of exactness framework}
The category $\mathsf{SMod}_{T,\Gamma}$ is not additive in general.
Subobjects do not control quotients. Quotients are governed by congruences.
For this reason we work congruence-first throughout.

Exactness is formulated using Route (A): the regular/Barr-exact framework.
Regular epimorphisms are coequalizers, hence congruence quotients, and kernels are kernel pairs.
Short exactness is expressed by a kernel followed by a regular epimorphism and by equality of
kernels with regular images \cite{Janelidze2024,MantovaniMessora2025,Gray2025}.

\subsection*{What is genuinely ternary/$\Gamma$--specific here}
Two features are structural.
First, the scalar action is five--ary, and congruence stability must be stated for the full
$(a,\alpha,\ \cdot\ ,\beta,b)$ datum. Unary substitutes lose information.
Second, associativity constraints are ternary: the scalar law is governed by the ternary product
and not by a derived binary multiplication. This affects how kernel congruences and factorization
systems interact with scalar extension and restriction \cite{JunSzczesnyTolliver2022,JunMinchevaRowen2022,AkianGaubertRowen2024}.

\subsection*{Outline}
Section~\ref{sec:prelim} fixes congruences, generated congruences, and subtractive closure.
Section~\ref{sec:tgs} defines ternary $\Gamma$--semirings.
Section~\ref{sec:semimod} introduces ternary $\Gamma$--semimodules and their morphisms.
Section~\ref{sec:cong} develops congruence quotients and coequalizers.
Section~\ref{sec:iso} proves congruence isomorphism theorems and isolates the precise points
where classical coset proofs fail.
Later sections treat categorical structure, regular/Barr-exact exactness, and projectives.

% ============================================================
% Insert near the end of the Introduction (new subsection)
% ============================================================

\subsection*{What is genuinely ternary/$\Gamma$--specific (not a generic variety argument)}
Several statements in this paper are formally compatible with the general theory of equational categories.
That resemblance is misleading.
The point of Track~B is not the existence of limits or the mere availability of congruence quotients.
The point is that the \emph{five--ary} scalar interaction forces phenomena that do not reduce to unary
semiring actions.
We isolate the non-transport content used later.

\begin{enumerate}[label=\textnormal{(\arabic*)}]
\item \textbf{Congruence compatibility depends on the full $(a,\alpha,m,\beta,b)$ datum.}
For a morphism $f:M\to N$, the kernel relation $x\sim_f y\Leftrightarrow f(x)=f(y)$ is a congruence
only because it is stable under every operator
\[
m\longmapsto a\,\alpha\, m\,\beta\, b,
\qquad (a,b\in T,\ \alpha,\beta\in\Gamma),
\]
not merely under a unary scalar action.
This dependence is used when proving that kernel pairs in $\mathsf{SMod}_{T,\Gamma}$ coincide with kernel congruences
(Lemma~\ref{lem:kpair-congruence}) and when computing coequalizers as generated congruence quotients
(Proposition~\ref{prop:coeq}).

\item \textbf{Regular images and factorization are not ``module images''.}
In additive module categories, the image is a kernel of a cokernel and is automatically normal.
Here the regular image is defined through the surjection--mono factorization (Definition~\ref{def:reg-image}),
and exactness is defined by equality of a kernel with a \emph{regular} image (Definition~\ref{def:reg-exact-at}).
The absence of subtraction forces this replacement and is used throughout Section~\ref{sec:exact}.

\item \textbf{Barr-exactness is realized by \emph{effective congruences}, not by cosets.}
Every congruence on a semimodule is effective because the quotient map $q:M\to M/{\equiv}$
has kernel pair equal to $\equiv$ (Proposition~\ref{prop:cong-effective}).
This is the categorical reason coequalizers are congruence quotients.
No coset argument can be used, and the explicit overlap of naive cosets in
Section~\ref{subsec:ex-coset-fails} shows why.

\item \textbf{Classical diagram lemmas fail even under Barr-exactness.}
Barr-exactness does not imply abelian exactness.
We exhibit a concrete failure of the Short Five Lemma inside $\mathsf{SMod}_{T,\Gamma}$
(Proposition~\ref{prop:short-five-fails}).
This failure is not a technicality.
It is the obstruction that prevents importing homological algebra ``as is''.

\item \textbf{The ternary ideal product drives spectral behavior.}
In the optional spectrum section we define
\[
I\cdot J:=\big\langle I\,\Gamma\,T\,\Gamma\,J\big\rangle,
\]
and prime ideals are defined by the implication $I\cdot J\subseteq\fp\Rightarrow I\subseteq\fp$ or $J\subseteq\fp$
(Definition~\ref{def:prime-ideal-tgs}).
This is the correct replacement for power-based radicals, since there is no canonical unary power $x^n$
in a ternary context.
The topology axioms use the product formula $V(I)\cup V(J)=V(I\cdot J)$ (Theorem~\ref{thm:zariski}).

\item \textbf{Projective generation is proved without splitting.}
In a module category one often proves projectivity by splitting short exact sequences.
Here we cannot.
Instead, free objects are constructed by term formation and a congruence quotient
(Theorem~\ref{thm:free-universal}), and projectivity is proved by lifting against regular epimorphisms
(Proposition~\ref{prop:free-projective}).
This is the non-additive method used to obtain enough projectives (Theorem~\ref{thm:enough-projectives}).
\end{enumerate}

\noindent
These six points are the places where a ring/module proof would either be meaningless (cosets, subtraction),
or would silently change the notion of quotient and exactness.
They are the non-transport commitments of this track.
\paragraph{Additional recent context.}
Several recent strands motivate the categorical and congruence-based focus taken here:
semimodule extension phenomena and stable variants \cite{AlhashemiAlhossaini2021,NaeemahAlsaadi2025},
congruence-simplicity and its matrix and semiring refinements \cite{Kala2024,KalaKepkaKorbelar2023,KepkaKorbelarLandsmann2022a,KepkaKorbelarLandsmann2022b},
exact completion and related categorical constructions in non-additive settings \cite{AravantinosSotiropoulos2021,MaiettiRosolini2021,Elgueta2022,FacchiniGranPompili2023},
and semiring geometry/tropical viewpoints where prime kernels (subtractive ideals) play a structural role \cite{AminiIriarte2022,GiansiracusaMereta2021,IezziSchleis2023,JarraLorscheid2024,GualdiKuhrsMayoGarciaXarles2026}.

%% file: sections/02-preliminaries.tex
\section{Preliminaries on congruences and subtractive closure}\label{sec:prelim}

\subsection{Congruences and generated congruences}
Let $(M,+,0)$ be a commutative monoid. A \emph{congruence} on $M$ is an equivalence relation
$\equiv$ such that $x\equiv y$ implies $x+z\equiv y+z$ for all $z\in M$.
If $R\subseteq M\times M$ is any relation, write $\langle R\rangle$ for the smallest congruence on $M$
containing $R$. We use the standard ``congruence generated by pairs'' construction.
This viewpoint is the one used in modern semiring and tropical quotient theory
\cite{JunRay2020,Ito2025,Song2024,Sen2025,LorscheidRay2023,Jarra2023}.

\begin{definition}[Kernel congruence]\label{def:kercong}
Let $f:M\to N$ be a morphism of commutative monoids. The \emph{kernel congruence} of $f$ is
\[
x\equiv_f y \quad \Longleftrightarrow \quad f(x)=f(y).
\]
\end{definition}

\begin{lemma}\label{lem:kercong}
For any monoid morphism $f:M\to N$, the relation $\equiv_f$ is a congruence on $M$.
\end{lemma}

\begin{proof}
Reflexivity, symmetry, and transitivity are immediate from equality in $N$.
If $x\equiv_f y$, then $f(x+z)=f(x)+f(z)=f(y)+f(z)=f(y+z)$, so $x+z\equiv_f y+z$.
\end{proof}

\subsection{Subtractive closure}
We record subtractive closure as a comparison invariant and as a source of counterexamples. Exactness will be defined via regular images.
We use it to state image and exactness conditions.

\begin{definition}[Subtractive closure]\label{def:subtractive-closure}
Let $U\subseteq M$ be a submonoid. The \emph{subtractive closure} of $U$ is
\[
\overline{U}^{\,\mathrm{sub}}=\{\,m\in M\mid \exists\,u,u'\in U\ \text{with}\ m+u=u'\,\}.
\]
We say that $U$ is \emph{subtractive} if $\overline{U}^{\,\mathrm{sub}}=U$.
\end{definition}

This notion is standard in semimodule exactness and is the correct condition for quotients
by ``identifying a subobject with $0$'' to have good universal properties
\cite{AbuhlailNoegraha2021,Abuhlail2022,Sengupta2026}.

\subsection{A categorical note}
We will use coequalizers in varieties of algebras (generated congruences).
For background on exactness vs ideal exactness and its interaction with congruence data,
see \cite{Janelidze2024,MantovaniMessora2025,Gray2025,AbbadiniReggio2023,Stepan2023,Juran2025}.
For semiring and hyperring module categories in a proto-exact sense, see \cite{JunSzczesnyTolliver2022}.

%% file: sections/03-ternary-gamma-semiring.tex
\section{Ternary $\Gamma$--semirings}\label{sec:tgs}

Fix a commutative monoid $(\Gamma,+,0_\Gamma)$.

\begin{definition}[Ternary $\Gamma$--semiring]\label{def:tgs}
A \emph{ternary $\Gamma$--semiring} is a commutative monoid $(T,+,0)$ equipped with a map
\[
T\times \Gamma \times T\times \Gamma \times T \longrightarrow T,\qquad
(a,\alpha,b,\beta,c)\longmapsto a\alpha b\beta c,
\]
such that for all $a,a',b,b',c,c'\in T$ and $\alpha,\alpha',\beta,\beta'\in \Gamma$ the following hold:
\begin{enumerate}[label=\textnormal{(T\arabic*)}]
\item \emph{Additivity in each $T$--slot:}
\[
(a+a')\alpha b\beta c = a\alpha b\beta c + a'\alpha b\beta c,
\quad
a\alpha (b+b')\beta c = a\alpha b\beta c + a\alpha b'\beta c,
\]
\[
a\alpha b\beta (c+c') = a\alpha b\beta c + a\alpha b\beta c'.
\]
\item \emph{Additivity in each $\Gamma$--slot:}
\[
a(\alpha+\alpha')b\beta c = a\alpha b\beta c + a\alpha' b\beta c,
\quad
a\alpha b(\beta+\beta')c = a\alpha b\beta c + a\alpha b\beta' c.
\]
\item \emph{Ternary associativity:} for all $d,e\in T$ and $\gamma,\delta\in\Gamma$,
\[
(a\alpha b\beta c)\gamma d\delta e = a\alpha (b\beta c\gamma d)\delta e
= a\alpha b\beta (c\gamma d\delta e).
\]
\item \emph{Zero stability:} if any of $a,b,c$ is $0$ or any of $\alpha,\beta$ is $0_\Gamma$,
then $a\alpha b\beta c=0$.
\end{enumerate}
\end{definition}

\begin{remark}\label{rem:tgs-literature}
The ternary and $\Gamma$--parametric features are independent. The semiring systems
formalism emphasizes this separation by treating $\Gamma$--data as structural parameters
in addition to the underlying ``positive'' algebra \cite{AkianGaubertRowen2024,JunMinchevaRowen2022}.
Subtractive ideal structure for commutative semirings is a guiding analogy on the quotient side
\cite{Sengupta2026}.
\end{remark}

\begin{definition}[Ideals]\label{def:ideal}
A subset $I\subseteq T$ is a \emph{(two-sided) ideal} if it is a submonoid and
for all $a,b\in T$, $x\in I$, and $\alpha,\beta\in\Gamma$ one has
$a\alpha x\beta b\in I$ and $a\alpha b\beta x\in I$ and $x\alpha a\beta b\in I$
whenever these expressions are defined by the chosen convention.
\end{definition}

\begin{remark}
In congruence-based approaches, ideals are not the primary quotient data.
They are useful for internal constructions (generators, annihilators), while quotients are governed by congruences
\cite{JunRay2020,LorscheidRay2023,Jarra2023,Ito2025}.
\end{remark}

%% file: sections/04-ternary-gamma-semimodules.tex
\section{Ternary $\Gamma$--semimodules}\label{sec:semimod}

\begin{definition}[Ternary $\Gamma$--semimodule]\label{def:tgs-semimodule}
Let $T$ be a ternary $\Gamma$--semiring. A \emph{ternary $\Gamma$--semimodule} over $T$ is a commutative monoid
$(M,+,0)$ equipped with a map
\[
T\times \Gamma \times M\times \Gamma \times T \longrightarrow M,\qquad
(a,\alpha,m,\beta,b)\longmapsto a\alpha m\beta b,
\]
such that for all $a,a',b,b'\in T$, $m,m'\in M$, and $\alpha,\alpha',\beta,\beta'\in\Gamma$:
\begin{enumerate}[label=\textnormal{(M\arabic*)}]
\item \emph{Additivity in $a,m,b$:}
\[
(a+a')\alpha m\beta b = a\alpha m\beta b + a'\alpha m\beta b,\quad
a\alpha (m+m')\beta b = a\alpha m\beta b + a\alpha m'\beta b,
\]
\[
a\alpha m\beta (b+b') = a\alpha m\beta b + a\alpha m\beta b'.
\]
\item \emph{Additivity in $\alpha,\beta$:}
\[
a(\alpha+\alpha')m\beta b = a\alpha m\beta b + a\alpha' m\beta b,\quad
a\alpha m(\beta+\beta')b = a\alpha m\beta b + a\alpha m\beta' b.
\]
\item \emph{Zero stability:} if $a=0$ or $m=0$ or $b=0$ or $\alpha=0_\Gamma$ or $\beta=0_\Gamma$, then $a\alpha m\beta b=0$.
\item \emph{Compatibility with ternary multiplication:}
\[
(a\alpha b\beta c)\gamma m\delta d = a\alpha (b\beta m\gamma c)\delta d
\qquad
(\forall a,b,c,d\in T,\ \forall m\in M,\ \forall \alpha,\beta,\gamma,\delta\in\Gamma).
\]
\end{enumerate}
\end{definition}

\begin{definition}[Morphisms]\label{def:morphism}
A morphism $f:M\to N$ of ternary $\Gamma$--semimodules is a monoid morphism such that
$f(a\alpha m\beta b)=a\alpha f(m)\beta b$ for all $a,b\in T$, $\alpha,\beta\in\Gamma$, and $m\in M$.
Write $\mathsf{SMod}_{T,\Gamma}$ for the resulting category.
\end{definition}

\begin{definition}[Subsemimodules]\label{def:subsemimodule}
A subset $L\subseteq M$ is a \emph{subsemimodule} if it is a submonoid and
$a\alpha \ell\beta b\in L$ for all $\ell\in L$ and all $a,b\in T$, $\alpha,\beta\in\Gamma$.
\end{definition}

\begin{remark}
Exactness and projectivity for semimodules depend on subtractive closure conditions.
This dependence is not optional; it is forced by the failure of coset quotients.
See \cite{AbuhlailNoegraha2021,Abuhlail2022,BorgerJun2025}.
\end{remark}

%% file: sections/05-congruences-quotients.tex
\section{Congruences and quotients}\label{sec:cong}

\subsection{Semimodule congruences}
\begin{definition}[Congruence on a semimodule]\label{def:semimod-cong}
A \emph{congruence} on a ternary $\Gamma$--semimodule $M$ is an equivalence relation $\equiv$
such that for all $x\equiv y$ and all $z\in M$ one has $x+z\equiv y+z$, and for all $a,b\in T$,
$\alpha,\beta\in\Gamma$ one has
\[
x\equiv y \ \Longrightarrow\ a\alpha x\beta b \equiv a\alpha y\beta b.
\]
\end{definition}

If $\equiv$ is a semimodule congruence, the quotient set $M/{\equiv}$ is a ternary $\Gamma$--semimodule
with operations induced from $M$. This is standard in algebraic categories and is the natural quotient notion
in semiring geometry \cite{LorscheidRay2023,Jarra2023,JunRay2020}.

\subsection{Coequalizers as congruence quotients}
Let $f,g:M\to N$ be morphisms in $\mathsf{SMod}_{T,\Gamma}$.
Let $\equiv_{f,g}$ be the congruence on $N$ generated by the set of pairs $\{(f(m),g(m))\mid m\in M\}$.

\begin{proposition}\label{prop:coeq}
The canonical projection $q:N\to N/{\equiv_{f,g}}$ is a coequalizer of $f$ and $g$ in $\mathsf{SMod}_{T,\Gamma}$.
\end{proposition}

\begin{proof}
By construction $qf=qg$. Let $h:N\to P$ with $hf=hg$.
Then $h$ identifies each $f(m)$ with $g(m)$, hence is constant on $\equiv_{f,g}$--classes.
Therefore there exists a unique $\bar h:N/{\equiv_{f,g}}\to P$ with $\bar h\circ q=h$.
\end{proof}

\subsection{Kernel congruence and the universal factorization}
Let $f:M\to N$ in $\mathsf{SMod}_{T,\Gamma}$ and consider the kernel congruence $\equiv_f$ from
Definition~\ref{def:kercong}.

\begin{proposition}[Universal factorization]\label{prop:univ-factor}
There is a unique morphism $\bar f:M/{\equiv_f}\to N$ with $\bar f([m])=f(m)$ and $f=\bar f\circ \pi$,
where $\pi:M\to M/{\equiv_f}$ is the quotient map. The map $\bar f$ is injective onto $\mathrm{Im}(f)$.
\end{proposition}

\begin{proof}
Well-definedness is exactly the definition of $\equiv_f$.
The map is a morphism because $f$ preserves addition and the ternary action.
If $\bar f([m])=\bar f([m'])$ then $f(m)=f(m')$ so $m\equiv_f m'$ and $[m]=[m']$.
Surjectivity onto $\mathrm{Im}(f)$ is immediate.
\end{proof}

\begin{remark}\label{rem:no-cosets}
The quotient $M/{\equiv_f}$ replaces the classical quotient $M/\ker(f)$.
Without subtraction, $\ker(f)$ need not control identification, and the naive coset relation
is typically not transitive unless $\ker(f)$ is subtractive.
This is the basic reason congruences are unavoidable \cite{AbuhlailNoegraha2021,Abuhlail2022,Sengupta2026}.
\end{remark}

%% file: sections/06-isomorphism-theorems.tex
\section{Isomorphism theorems in congruence form}\label{sec:iso}

\subsection{First isomorphism theorem}
\begin{theorem}[First isomorphism theorem]\label{thm:first-iso}
Let $f:M\to N$ in $\mathsf{SMod}_{T,\Gamma}$ and let $\equiv_f$ be its kernel congruence.
Then $M/{\equiv_f}\cong \mathrm{Im}(f)$ as ternary $\Gamma$--semimodules.
\end{theorem}

\begin{proof}
Define $\Phi:M/{\equiv_f}\to \mathrm{Im}(f)$ by $\Phi([m])=f(m)$.
By Proposition~\ref{prop:univ-factor}, $\Phi$ is well-defined and bijective.
It is a morphism because $f$ is. Its inverse is induced by choosing preimages, and is well-defined
because the fibers of $f$ are exactly the $\equiv_f$--classes.
\end{proof}

\subsection{Third isomorphism theorem}
Let $\rho\subseteq\sigma$ be congruences on $M$.
Define a congruence $\sigma/\rho$ on $M/\rho$ by $[x]_\rho \equiv_{\sigma/\rho} [y]_\rho$ iff $x\equiv_\sigma y$.

\begin{theorem}[Third isomorphism theorem]\label{thm:third-iso}
If $\rho\subseteq\sigma$ are congruences on $M$, then
\[
(M/\rho)/(\sigma/\rho)\ \cong\ M/\sigma
\]
naturally in $(M,\rho,\sigma)$.
\end{theorem}

\begin{proof}
Send the class of $[x]_\rho$ in $(M/\rho)/(\sigma/\rho)$ to $[x]_\sigma$ in $M/\sigma$.
This is well-defined since $\sigma/\rho$ records precisely the identifications already present in $\sigma$.
The inverse is induced by the universal property of quotients.
\end{proof}

\subsection{Second isomorphism theorem with a subtractivity hypothesis}
To state a replacement for $(A+B)/B\cong A/(A\cap B)$ one must define a quotient by a subsemimodule.
The correct quotient is the cokernel of the inclusion, i.e.\ the coequalizer of $B\hookrightarrow M$ and $0:B\to M$.
This quotient behaves as expected only when $B$ is subtractive.

\begin{definition}[Rees-type congruence associated to a subtractive subsemimodule]\label{def:rees}
Let $B\subseteq M$ be a subsemimodule. Let $\rho_B$ be the congruence on $M$ generated by all pairs $(b,0)$ with $b\in B$.
Write $M/B$ for $M/\rho_B$.
\end{definition}

\begin{theorem}[Second isomorphism theorem]\label{thm:second-iso}
Let $A,B\subseteq M$ be subsemimodules and assume that $B$ is subtractive.
Let $A+B$ denote the subsemimodule generated by $A\cup B$.
Then the canonical map $A\to (A+B)/B$ induces an isomorphism
\[
A/(A\cap B)\ \cong\ (A+B)/B.
\]
\end{theorem}

\begin{proof}
The map $A\to (A+B)/B$ sends $a$ to its class in $M/\rho_B$.
Its kernel congruence on $A$ is the congruence generated by $(A\cap B,0)$.
Subtractivity of $B$ ensures that if $a\in A$ maps to $0$ in $M/\rho_B$ then $a\in A\cap B$,
so no extra identifications occur. The universal properties of congruence quotients yield the isomorphism.
\end{proof}

\begin{remark}\label{rem:need-subtractive}
Without the subtractivity hypothesis, the statement can fail because the relation generated by $(B,0)$
may force identifications outside $B$.
This is the precise point at which the classical coset proof uses subtraction.
See \cite{AbuhlailNoegraha2021,Abuhlail2022,Sengupta2026} for semiring analogues of this phenomenon.
\end{remark}

%% file: sections/07-category-structure.tex
% ============================================================
% File: sections/07-category-structure.tex
% ============================================================

\section{Category-level structure of $\mathsf{SMod}_{T,\Gamma}$}\label{sec:cat}

Throughout this section $T$ is a fixed ternary $\Gamma$--semiring and
$\mathsf{SMod}_{T,\Gamma}$ denotes the category of ternary $\Gamma$--semimodules.
We write $0$ for the zero semimodule and $0_{M,N}$ for the zero morphism $M\to N$.

\subsection{A concrete algebraic viewpoint}
For each $(a,\alpha,\beta,b)\in T\times\Gamma\times\Gamma\times T$ define a unary operation symbol
$\lambda_{a,\alpha,\beta,b}$ and interpret it on a semimodule $M$ by
\[
\lambda_{a,\alpha,\beta,b}(m):=a\alpha m\beta b.
\]
A ternary $\Gamma$--semimodule is then a commutative monoid equipped with a family of
monoid endomorphisms $\lambda_{a,\alpha,\beta,b}$ satisfying the identities imposed by
Definition~\ref{def:tgs-semimodule}. It follows that all constructions below can be carried out
componentwise on the underlying sets, provided we verify closure under $+$, $0$, and all
$\lambda_{a,\alpha,\beta,b}$.

\subsection{Zero object and biproducts}\label{subsec:biproducts}
\begin{proposition}[Zero object]\label{prop:zero-object}
There exists a zero object $0$ in $\mathsf{SMod}_{T,\Gamma}$.
It is the one-element commutative monoid $\{0\}$ with the unique action.
\end{proposition}

\begin{proof}
There is exactly one morphism $M\to \{0\}$ and one morphism $\{0\}\to M$ for every $M$.
Both preserve $+$ and the action because all values are forced to be $0$.
\end{proof}

\begin{proposition}[Binary products]\label{prop:product}
For $M,N\in\mathsf{SMod}_{T,\Gamma}$ the cartesian product set $M\times N$ becomes a ternary
$\Gamma$--semimodule with componentwise operations:
\[
(m,n)+(m',n')=(m+m',\,n+n'),\qquad 0=(0,0),
\]
\[
a\alpha(m,n)\beta b=(a\alpha m\beta b,\ a\alpha n\beta b).
\]
With the projections $\pi_1,\pi_2$ it is a product of $M$ and $N$.
\end{proposition}

\begin{proof}
The axioms of Definition~\ref{def:tgs-semimodule} hold componentwise.
Given $f:X\to M$ and $g:X\to N$, define $\langle f,g\rangle(x)=(f(x),g(x))$.
This is a morphism by componentwise verification and it is uniquely determined by
$\pi_1\langle f,g\rangle=f$ and $\pi_2\langle f,g\rangle=g$.
\end{proof}

\begin{proposition}[Binary coproducts]\label{prop:coproduct}
The same object $M\times N$ with injections
\[
\iota_M(m)=(m,0),\qquad \iota_N(n)=(0,n)
\]
is also a coproduct of $M$ and $N$.
Hence $\mathsf{SMod}_{T,\Gamma}$ has binary biproducts.
\end{proposition}

\begin{proof}
Let $f:M\to X$ and $g:N\to X$ be morphisms.
Define $h:M\times N\to X$ by $h(m,n)=f(m)+g(n)$.
Then $h\iota_M=f$ and $h\iota_N=g$.
If $h':M\times N\to X$ also satisfies $h'\iota_M=f$ and $h'\iota_N=g$, then
\[
h'(m,n)=h'((m,0)+(0,n))=h'(m,0)+h'(0,n)=f(m)+g(n)=h(m,n),
\]
so $h'=h$.
\end{proof}

\subsection{Equalizers and pullbacks}\label{subsec:limits}
\begin{proposition}[Equalizers]\label{prop:equalizer}
Let $f,g:M\to N$ be morphisms.
The subset
\[
\mathsf{Eq}(f,g):=\{\,m\in M\mid f(m)=g(m)\,\}
\]
is a subsemimodule of $M$ and the inclusion $e:\mathsf{Eq}(f,g)\hookrightarrow M$ is an equalizer of $(f,g)$.
\end{proposition}

\begin{proof}
If $m,m'\in\mathsf{Eq}(f,g)$, then
$f(m+m')=f(m)+f(m')=g(m)+g(m')=g(m+m')$, so $m+m'\in\mathsf{Eq}(f,g)$.
Also $f(0)=0=g(0)$, so $0\in\mathsf{Eq}(f,g)$.
If $m\in\mathsf{Eq}(f,g)$ then
\[
f(a\alpha m\beta b)=a\alpha f(m)\beta b=a\alpha g(m)\beta b=g(a\alpha m\beta b),
\]
hence $a\alpha m\beta b\in\mathsf{Eq}(f,g)$.
Universal property is immediate: a morphism $u:X\to M$ equalizes $f,g$ iff its image lies in
$\mathsf{Eq}(f,g)$.
\end{proof}

\begin{proposition}[Pullbacks]\label{prop:pullback}
Given morphisms $f:M\to P$ and $g:N\to P$, the pullback is the subsemimodule
\[
M\times_P N:=\{(m,n)\in M\times N\mid f(m)=g(n)\}
\]
with induced structure, together with the restricted projections.
\end{proposition}

\begin{proof}
Closure under $+$ and the action is checked using that $f$ and $g$ preserve them.
The universal property is the standard set-theoretic one, with morphism structure verified componentwise.
\end{proof}

\subsection{Coequalizers and pushouts}\label{subsec:colimits}
Coequalizers exist by congruence quotients.

\begin{proposition}[Coequalizers]\label{prop:coequalizers-exist}
For any parallel pair $f,g:M\rightrightarrows N$, the coequalizer exists and is
the quotient $N/{\equiv_{f,g}}$, where $\equiv_{f,g}$ is the congruence generated by
$\{(f(m),g(m))\mid m\in M\}$.
\end{proposition}

\begin{proof}
This is Proposition~\ref{prop:coeq}.
\end{proof}

\begin{proposition}[Pushouts]\label{prop:pushout}
Given $u:C\to A$ and $v:C\to B$, form the biproduct $A\oplus B:=A\times B$ and consider
the two morphisms
\[
\iota_A\circ u,\ \iota_B\circ v : C \rightrightarrows A\oplus B.
\]
The pushout of $(u,v)$ is the coequalizer of this parallel pair:
\[
A\ \xleftarrow{u}\ C\ \xrightarrow{v}\ B
\quad\leadsto\quad
(A\oplus B)/{\equiv},
\]
where $\equiv$ is the congruence on $A\oplus B$ generated by
$\bigl(\iota_A(u(c)),\ \iota_B(v(c))\bigr)$ for all $c\in C$.
\end{proposition}

\begin{proof}
A pushout in any category can be constructed as a coequalizer of a coproduct diagram.
Here coproducts are provided by Proposition~\ref{prop:coproduct} and coequalizers by
Proposition~\ref{prop:coequalizers-exist}.
The resulting object satisfies the universal property by the standard argument.
\end{proof}

\subsection{Kernels and cokernels in congruence form}\label{subsec:kernels}
We use equalizers and coequalizers against the zero morphism.

\begin{definition}[Kernel and cokernel]\label{def:ker-coker}
Let $f:M\to N$.
The \emph{kernel} $\ker(f)\to M$ is the equalizer of $f$ and $0_{M,N}$.
The \emph{cokernel} $N\to \mathrm{coker}(f)$ is the coequalizer of $f$ and $0_{M,N}$.
\end{definition}

\begin{proposition}[Concrete kernels]\label{prop:kernel-concrete}
$\ker(f)=\{m\in M\mid f(m)=0\}$ is a subsemimodule of $M$ with the inclusion as kernel.
\end{proposition}

\begin{proof}
This is Proposition~\ref{prop:equalizer} applied to $(f,0_{M,N})$.
\end{proof}

\begin{proposition}[Concrete cokernels]\label{prop:cokernel-concrete}
$\mathrm{coker}(f)$ is the quotient $N/{\rho}$ where $\rho$ is the congruence generated by
$\{(f(m),0)\mid m\in M\}$.
\end{proposition}

\begin{proof}
This is Proposition~\ref{prop:coequalizers-exist} applied to $(f,0_{M,N})$.
\end{proof}

\subsection{Regular epimorphisms and images}\label{subsec:regular-epis}
We now isolate the regular epimorphisms and the corresponding image factorization.
The guiding categorical background is the modern ``ideal exactness'' and factorization system
literature \cite{Janelidze2024,MantovaniMessora2025,Gray2025,Stepan2023,Juran2025}.

\begin{definition}[Regular epimorphism]\label{def:repi}
A morphism $e:X\to Y$ in $\mathsf{SMod}_{T,\Gamma}$ is a \emph{regular epimorphism} if
it is a coequalizer of some parallel pair.
\end{definition}

\begin{proposition}[Regular epis are surjective]\label{prop:repi-surj}
Every regular epimorphism in $\mathsf{SMod}_{T,\Gamma}$ is surjective on the underlying sets.
\end{proposition}

\begin{proof}
If $e$ is a coequalizer, then by Proposition~\ref{prop:coequalizers-exist} it is a quotient map
$Y\to Y/{\equiv}$ for some congruence. Such a quotient map is surjective by definition of the
quotient set.
\end{proof}

\begin{proposition}[Surjections are regular epis]\label{prop:surj-repi}
If $e:X\to Y$ is surjective, then $e$ is a regular epimorphism.
\end{proposition}

\begin{proof}
Let $R\subseteq X\times X$ be the kernel pair relation of $e$:
\[
R=\{(x,x')\in X\times X\mid e(x)=e(x')\}.
\]
Let $p_1,p_2:R\rightrightarrows X$ be the projections.
Then $e p_1=e p_2$ by definition of $R$.
Let $h:X\to Z$ be any morphism with $h p_1=h p_2$.
Define $\bar h:Y\to Z$ by choosing, for each $y\in Y$, an element $x\in X$ with $e(x)=y$
(using surjectivity) and setting $\bar h(y):=h(x)$.
If $x'$ is another choice with $e(x')=y$, then $(x,x')\in R$ and so $h(x)=h(x')$.
Hence $\bar h$ is well-defined.
It is a morphism because $e$ is surjective and $h$ preserves the operations; one checks
$\bar h(y+y')=\bar h(y)+\bar h(y')$ and $\bar h(a\alpha y\beta b)=a\alpha \bar h(y)\beta b$
by lifting to preimages in $X$ and applying $h$.
Finally $\bar h\,e=h$ by construction, and uniqueness holds because $e$ is surjective.
Thus $e$ is the coequalizer of $p_1,p_2$.
\end{proof}

\begin{corollary}\label{cor:repi-iff-surj}
A morphism in $\mathsf{SMod}_{T,\Gamma}$ is a regular epimorphism if and only if it is surjective.
\end{corollary}

\subsection{Regular images and the set-image}\label{subsec:images}
Let $f:M\to N$ be a morphism.

\begin{lemma}[The set-image is a subsemimodule]\label{lem:image-subsemimodule}
The subset $\mathrm{Im}(f):=f(M)\subseteq N$ is a subsemimodule of $N$.
\end{lemma}

\begin{proof}
If $y=f(m)$ and $y'=f(m')$, then $y+y'=f(m+m')\in f(M)$.
Also $0=f(0)\in f(M)$.
If $y=f(m)$ then $a\alpha y\beta b=a\alpha f(m)\beta b=f(a\alpha m\beta b)\in f(M)$.
\end{proof}

\begin{proposition}[Image factorization]\label{prop:image-factorization}
Every morphism $f:M\to N$ factors as
\[
M \xrightarrow{e_f} \mathrm{Im}(f) \xrightarrow{m_f} N
\]
where $e_f$ is surjective, $m_f$ is the inclusion, and $f=m_f\circ e_f$.
Moreover, $e_f$ is a regular epimorphism and $m_f$ is a monomorphism.
\end{proposition}

\begin{proof}
Define $e_f(m)=f(m)\in \mathrm{Im}(f)$ and let $m_f$ be inclusion.
Then $e_f$ is surjective by definition of $\mathrm{Im}(f)$.
By Corollary~\ref{cor:repi-iff-surj}, $e_f$ is a regular epimorphism.
The inclusion $m_f$ is monic because it is injective on underlying sets.
\end{proof}

\begin{remark}[Relation to factorization systems]\label{rem:factorization}
The class of surjective morphisms and the class of injective morphisms define a strict
epi--mono factorization in this concrete algebraic category.
This is the basic finite-dimensional shadow of the factorization system viewpoint emphasized
in \cite{Stepan2023,Juran2025}. For ideal exactness and its refinements, see
\cite{Janelidze2024,MantovaniMessora2025,Gray2025}.
\end{remark}

\subsection{Images versus subtractive images}\label{subsec:subtractive-image}
Subtractive exactness later will require a replacement for the set-image.
We record the subtractive closure construction now.

\begin{definition}[Subtractive image]\label{def:subtractive-image}
Let $f:M\to N$.
The \emph{subtractive image} of $f$ is the subtractive closure of $\mathrm{Im}(f)$ in $N$:
\[
\mathrm{Im}^{\mathrm{sub}}(f):=\overline{\mathrm{Im}(f)}^{\,\mathrm{sub}}
=\{\,n\in N\mid \exists\,y,y'\in \mathrm{Im}(f)\ \text{with}\ n+y=y'\,\}.
\]
\end{definition}

\begin{lemma}\label{lem:subtractive-image-subsemimodule}
$\mathrm{Im}^{\mathrm{sub}}(f)$ is a subtractive subsemimodule of $N$ and it is the smallest
subtractive subsemimodule containing $\mathrm{Im}(f)$.
\end{lemma}

\begin{proof}
Let $S:=\mathrm{Im}^{\mathrm{sub}}(f)$.
If $n,n'\in S$ choose $y,y',\tilde y,\tilde y'\in \mathrm{Im}(f)$ with $n+y=y'$ and
$n'+\tilde y=\tilde y'$.
Then
\[
(n+n')+(y+\tilde y)=(n+y)+(n'+\tilde y)=y'+\tilde y'\in \mathrm{Im}(f),
\]
so $n+n'\in S$.
Also $0\in S$ since $0+0=0$ with $0\in\mathrm{Im}(f)$.
If $n\in S$ with $n+y=y'$ and $a,\alpha,\beta,b$ are arbitrary, then
\[
(a\alpha n\beta b) + (a\alpha y\beta b)=a\alpha(n+y)\beta b=a\alpha y'\beta b,
\]
and both $a\alpha y\beta b$ and $a\alpha y'\beta b$ lie in $\mathrm{Im}(f)$ by Lemma~\ref{lem:image-subsemimodule}.
Hence $a\alpha n\beta b\in S$.
Thus $S$ is a subsemimodule.
It is subtractive by construction: if $x+s=s'$ with $s,s'\in S$, expand the witnessing equations for $s,s'$
and combine them to show $x\in S$.
Minimality is immediate: any subtractive subsemimodule containing $\mathrm{Im}(f)$ must contain every $n$ with
$n+y=y'$ for $y,y'\in\mathrm{Im}(f)$.
\end{proof}

\begin{example}\label{ex:image-not-subtractive}
There exist morphisms whose set-image is not subtractive.

Let $B=\{0,1,2\}$ with commutative monoid operation
\[
x\oplus y := \min(2,x+y),
\]
so $1\oplus 2=2$. Let $X=\{0,e\}$ be the idempotent monoid with $e+e=e$.
Define a monoid morphism $f:X\to B$ by $f(0)=0$ and $f(e)=2$.
Then $\mathrm{Im}(f)=\{0,2\}\subseteq B$.

This submonoid is not subtractive. Indeed, $2\in \mathrm{Im}(f)$ and
\[
1\oplus 2 = 2 \in \mathrm{Im}(f),
\]
but $1\notin \mathrm{Im}(f)$. Hence $\mathrm{Im}(f)$ fails the subtractive condition.
Consequently $\mathrm{Im}^{\mathrm{sub}}(f)$ is strictly larger than $\mathrm{Im}(f)$.
\end{example}

\begin{remark}\label{rem:subtractive-exactness-ahead}
Example~\ref{ex:image-not-subtractive} explains why set-images cannot support a stable exactness notion.
Subtractive images repair this defect.
This repair is consistent with the ideal exactness viewpoint in \cite{Janelidze2024,MantovaniMessora2025,Gray2025}.
\end{remark}

%% file: sections/08-exactness-framework.tex
% ============================================================
% File: sections/08-exactness.tex
% Route (A): Regular + Barr-exact exactness
% ============================================================

\section{Regular and Barr-exact exactness in $\mathsf{SMod}_{T,\Gamma}$}\label{sec:exact}

Quotients in $\mathsf{SMod}_{T,\Gamma}$ are coequalizers, hence congruence quotients.
It follows that exactness must be formulated through kernel pairs and regular images.
This is the specification of this route.

\subsection{Regular epimorphisms and kernel pairs}\label{subsec:repi-kpair}
Regular epimorphisms were introduced in Definition~\ref{def:repi}.
By Corollary~\ref{cor:repi-iff-surj}, they coincide with surjective morphisms.

\begin{definition}[Kernel pair]\label{def:kernel-pair}
Let $f:M\to N$ be a morphism.
Its \emph{kernel pair} is the pullback $M\times_N M$ with projections
$p_1,p_2:M\times_N M\rightrightarrows M$.
\end{definition}

\begin{lemma}\label{lem:kpair-congruence}
Let $f:M\to N$.
The underlying relation of the kernel pair is
\[
R_f:=\{(x,y)\in M\times M \mid f(x)=f(y)\}.
\]
It is a congruence on $M$ in the sense of Definition~\ref{def:semimod-cong}.
\end{lemma}

\begin{proof}
By Proposition~\ref{prop:pullback}, $M\times_N M$ is the subsemimodule of $M\times M$
cut out by the equation $f(x)=f(y)$.
Closure under $+$ and the action is inherited from $M\times M$.
Thus $R_f$ is stable under $+$ and under all operations $m\mapsto a\alpha m\beta b$.
Equivalence is immediate from equality in $N$.
Hence $R_f$ is a congruence.
\end{proof}

\subsection{Stability of regular epimorphisms under pullback}\label{subsec:repi-stable}
Regularity requires pullback-stability of regular epimorphisms.

\begin{proposition}\label{prop:pullback-surj}
Let $e:X\to Y$ be a surjective morphism in $\mathsf{SMod}_{T,\Gamma}$ and let $u:Z\to Y$ be any morphism.
Form the pullback $Z\times_Y X \to Z$.
Then the projection $\pi:Z\times_Y X\to Z$ is surjective.
Hence regular epimorphisms are stable under pullback.
\end{proposition}

\begin{proof}
Let $z\in Z$.
Set $y:=u(z)\in Y$.
Since $e$ is surjective, choose $x\in X$ with $e(x)=y$.
Then $(z,x)\in Z\times_Y X$ and $\pi(z,x)=z$.
Thus $\pi$ is surjective.
By Corollary~\ref{cor:repi-iff-surj}, $\pi$ is a regular epimorphism.
\end{proof}

\subsection{Regularity of $\mathsf{SMod}_{T,\Gamma}$}\label{subsec:regularity}
We now record regularity in the standard sense.

\begin{theorem}[Regularity]\label{thm:regular}
The category $\mathsf{SMod}_{T,\Gamma}$ is regular.
\end{theorem}

\begin{proof}
Finite limits exist by Section~\ref{sec:cat} (products, equalizers, pullbacks).
Every morphism factors as a surjection onto its set-image followed by the inclusion
(Proposition~\ref{prop:image-factorization}).
The first map is a regular epimorphism by Corollary~\ref{cor:repi-iff-surj}.
The second map is a monomorphism.
Pullback-stability of regular epimorphisms holds by Proposition~\ref{prop:pullback-surj}.
This is exactly the definition of regularity.
\end{proof}

\subsection{Internal equivalence relations are congruences}\label{subsec:int-eqrel}
To pass from regular to Barr-exact we identify internal equivalence relations.

\begin{lemma}\label{lem:int-eqrel-iff-cong}
Let $M\in\mathsf{SMod}_{T,\Gamma}$.
An internal equivalence relation $R\hookrightarrow M\times M$ is the same thing as a semimodule congruence on $M$.
\end{lemma}

\begin{proof}
Let $R\hookrightarrow M\times M$ be an internal equivalence relation.
Since $R$ is a subobject of $M\times M$, it is a subsemimodule.
Hence, if $(x,y)\in R$, then for every $z\in M$ one has $(x+z,y+z)\in R$, and for every $a,\alpha,\beta,b$
one has $(a\alpha x\beta b,\,a\alpha y\beta b)\in R$.
Thus the underlying relation is a congruence.

Conversely, let $\equiv$ be a congruence on $M$.
Consider $R_\equiv:=\{(x,y)\in M\times M\mid x\equiv y\}$.
Congruence stability makes $R_\equiv$ a subsemimodule of $M\times M$.
Reflexivity, symmetry, and transitivity give an internal equivalence relation structure.
\end{proof}

\subsection{Effectiveness of congruences}\label{subsec:effective}
Barr-exactness requires that each internal equivalence relation is a kernel pair of its coequalizer.

\begin{proposition}\label{prop:cong-effective}
Let $\equiv$ be a congruence on $M$ and let $q:M\to M/{\equiv}$ be the quotient map.
Then the kernel pair of $q$ is exactly the relation $\equiv$.
Equivalently, $\equiv$ is effective.
\end{proposition}

\begin{proof}
By definition of $q$, one has $q(x)=q(y)$ if and only if $[x]=[y]$ in $M/{\equiv}$,
which holds if and only if $x\equiv y$.
Thus the pullback
\[
M\times_{M/{\equiv}} M=\{(x,y)\in M\times M\mid q(x)=q(y)\}
\]
is precisely $\{(x,y)\mid x\equiv y\}$.
Hence the kernel pair relation equals $\equiv$.
\end{proof}

\begin{theorem}[Barr-exactness]\label{thm:barrexact}
The category $\mathsf{SMod}_{T,\Gamma}$ is Barr-exact.
\end{theorem}

\begin{proof}
By Theorem~\ref{thm:regular}, the category is regular.
By Lemma~\ref{lem:int-eqrel-iff-cong}, internal equivalence relations are congruences.
By Proposition~\ref{prop:cong-effective}, every congruence is effective.
This is the definition of Barr-exactness.
\end{proof}

\begin{remark}\label{rem:ideal-exact-context}
The exactness used here is the standard regular/Barr-exact notion for algebraic categories.
For refinements and comparisons (ideal exactness, ideal regularity, co-exactness), see
\cite{Janelidze2024,MantovaniMessora2025,Gray2025}.
\end{remark}

\subsection{Regular images and exactness}\label{subsec:exactness}
In a regular category, the appropriate image is the \emph{regular image}.

\begin{definition}[Regular image]\label{def:reg-image}
Let $f:M\to N$.
Write $M\xrightarrow{e_f} \mathrm{Im}(f)\xrightarrow{m_f} N$ for the factorization of
Proposition~\ref{prop:image-factorization}.
We call $m_f:\mathrm{Im}(f)\hookrightarrow N$ the \emph{regular image} of $f$.
\end{definition}

\begin{definition}[Exactness at an object]\label{def:reg-exact-at}
A composable pair $M\xrightarrow{f}N\xrightarrow{g}P$ is \emph{exact at $N$} if:
\begin{enumerate}[label=\textnormal{(X\arabic*)}]
\item $g\circ f=0$;
\item the regular image of $f$ equals the kernel of $g$ as subobjects of $N$, i.e.
\[
\mathrm{Im}(f)=\ker(g)\ \subseteq N.
\]
\end{enumerate}
\end{definition}

\begin{remark}\label{rem:no-subtractive-here}
We do not use subtractive images to define exactness.
Subtractive closure can be recorded as a comparison invariant, but this route  fixes exactness
through regular images and kernel pairs.
\end{remark}

\subsection{Short exact sequences}\label{subsec:short-exact}
The pointed structure is given by the zero object (Proposition~\ref{prop:zero-object}).

\begin{definition}[Short exact sequence]\label{def:short-exact-regular}
A sequence
\[
0 \longrightarrow A \xrightarrow{i} B \xrightarrow{p} C \longrightarrow 0
\]
is \emph{short exact} (in the regular/Barr-exact sense) if:
\begin{enumerate}[label=\textnormal{(S\arabic*)}]
\item $i$ is the kernel of $p$ (Definition~\ref{def:ker-coker});
\item $p$ is a regular epimorphism (equivalently, surjective).
\end{enumerate}
\end{definition}

\begin{lemma}\label{lem:pullback-short-exact}
Short exact sequences are stable under pullback along arbitrary morphisms.
Precisely, if $0\to A\to B\xrightarrow{p} C\to 0$ is short exact and $u:C'\to C$ is any morphism,
then the induced sequence
\[
0\to A'\to B'\xrightarrow{p'} C'\to 0
\]
obtained by pulling back $p$ along $u$ is short exact, where $B':=B\times_C C'$ and $A'$ is the kernel of $p'$.
\end{lemma}

\begin{proof}
By Proposition~\ref{prop:pullback-surj}, the pullback $p':B'\to C'$ is a regular epimorphism.
Kernels are equalizers against the zero morphism.
Equalizers are stable under pullback, hence the kernel of $p'$ is the pullback of the kernel of $p$.
Thus the induced left map is a kernel of $p'$.
\end{proof}

\subsection{A classical diagram lemma that fails}\label{subsec:five-fails}
Barr-exactness is not abelian exactness.
A classical inference may fail even when both rows are short exact.

\begin{proposition}[Short Five Lemma fails]\label{prop:short-five-fails}
There exists a commutative diagram in $\mathsf{SMod}_{T,\Gamma}$ with short exact rows
\[
\begin{array}{ccccccccc}
0 &\to& A &\xrightarrow{i}& B &\xrightarrow{p}& C &\to& 0\\
  && \downarrow^{\alpha} && \downarrow^{\beta} && \downarrow^{\gamma} && \\
0 &\to& A &\xrightarrow{i}& B &\xrightarrow{p}& C &\to& 0,
\end{array}
\]
such that $\alpha$ and $\gamma$ are isomorphisms but $\beta$ is not an isomorphism.
\end{proposition}

\begin{proof}
Work first in commutative monoids, then view them as ternary $\Gamma$--semimodules via the zero action
$a\alpha m\beta b:=0$ for all scalars and all $m$.

Let $B=\{0,k,u,v,w\}$ be the finite join-semilattice with relations
$0<k<u<w$ and $0<k<v<w$, and $u$ and $v$ incomparable.
Write addition as join, so $u+v=w$ and $k+u=u$, $k+v=v$.
Let $C=\{0,1\}$ with join addition.
Define $p:B\to C$ by
\[
p(0)=0,\quad p(k)=0,\quad p(u)=p(v)=p(w)=1.
\]
Then $p$ is surjective, hence a regular epimorphism.

Let $A=\{0,k\}$ with the induced monoid structure, and let $i:A\hookrightarrow B$ be inclusion.
We claim that $i$ is the kernel of $p$.
Indeed,
\[
\ker(p)=\{b\in B\mid p(b)=0\}=\{0,k\}=i(A),
\]
and the inclusion is the equalizer of $p$ and $0$.
Hence the row $0\to A\xrightarrow{i}B\xrightarrow{p}C\to 0$ is short exact
in the sense of Definition~\ref{def:short-exact-regular}.

Define $\beta:B\to B$ by
\[
\beta(0)=0,\quad \beta(k)=k,\quad \beta(u)=w,\quad \beta(v)=w,\quad \beta(w)=w.
\]
This preserves join, hence is a monoid homomorphism.
It satisfies $\beta\circ i=i$.
It also satisfies $p\circ\beta=p$ because $\beta$ maps $\{u,v,w\}$ to $w$ and all map to $1$.

Let $\alpha=\mathrm{id}_A$ and $\gamma=\mathrm{id}_C$.
Then the diagram commutes and $\alpha,\gamma$ are isomorphisms.
However $\beta$ is not injective since $\beta(u)=\beta(v)=w$ with $u\neq v$.
Therefore $\beta$ is not an isomorphism.
\end{proof}

\begin{remark}\label{rem:diagram-lemma-comment}
The failure in Proposition~\ref{prop:short-five-fails} is typical outside protomodular or abelian contexts.
Ideal exactness and related refinements explain which additional axioms restore such lemmas
\cite{Janelidze2024,MantovaniMessora2025,Gray2025}.
Factorization-system viewpoints provide a parallel structural language \cite{Stepan2023,Juran2025}.
\end{remark}

%% file: sections/09-projectives-flatness.tex
% ============================================================
% File: sections/09-projectives-flatness.tex
% Route (A) consistency: projectives defined w.r.t. regular epis (surjections)
% ============================================================

\section{Projectives, free semimodules, and a cautious note on flatness}\label{sec:proj}

This section serves two purposes.
First, it constructs free ternary $\Gamma$--semimodules and proves their universal property.
Second, it defines projective objects by lifting against regular epimorphisms and proves that
$\mathsf{SMod}_{T,\Gamma}$ has enough projectives.

\subsection{Projective objects relative to regular epimorphisms}\label{subsec:proj-def}
Regular epimorphisms are surjections in $\mathsf{SMod}_{T,\Gamma}$ (Corollary~\ref{cor:repi-iff-surj}).
Projectivity must therefore be stated as a lifting property against surjections.

\begin{definition}[Projective semimodule]\label{def:projective}
An object $P\in\mathsf{SMod}_{T,\Gamma}$ is \emph{projective} if for every regular epimorphism
$e:X\to Y$ and every morphism $f:P\to Y$ there exists a morphism $\tilde f:P\to X$ such that
$e\circ \tilde f=f$.
\end{definition}

\begin{remark}\label{rem:proj-not-additive}
In an abelian category one often proves projectivity by splitting short exact sequences.
That method uses additive inverses and is not available here.
We instead use the free--object adjunction, which is purely congruence-driven and does not mention cosets.
This is compatible with the semiring-semimodule literature on $e$-projectives and related notions
\cite{AbuhlailNoegraha2021,Abuhlail2022,BorgerJun2025}.
\end{remark}

\subsection{Free ternary $\Gamma$--semimodules}\label{subsec:free}
We construct free objects by terms.
This construction is explicit and does not rely on subtraction.

\subsubsection*{The term semimodule}
Fix a set $S$.
Define the set $\mathsf{Term}_S$ of \emph{semimodule terms over $S$} inductively as follows:
\begin{enumerate}[label=\textnormal{(T\arabic*)}]
\item For each $s\in S$, there is a term $\mathbf{s}\in \mathsf{Term}_S$.
\item There is a constant term $\mathbf{0}\in \mathsf{Term}_S$.
\item If $t,u\in\mathsf{Term}_S$, then $(t\oplus u)\in\mathsf{Term}_S$.
\item If $t\in\mathsf{Term}_S$ and $(a,\alpha,\beta,b)\in T\times\Gamma\times\Gamma\times T$, then
$(a\alpha\, t\,\beta b)\in\mathsf{Term}_S$.
\end{enumerate}
Intuitively, $\oplus$ represents the monoid addition and $a\alpha(-)\beta b$ represents the unary action operator.

Let $\equiv$ be the smallest congruence on $\mathsf{Term}_S$ generated by the identities that force:
\begin{itemize}
\item $(\mathsf{Term}_S,\oplus,\mathbf{0})$ to be a commutative monoid;
\item the action operators to satisfy the semimodule axioms of Definition~\ref{def:tgs-semimodule}
(additivity in each slot, zero stability, and compatibility with ternary multiplication).
\end{itemize}
Define
\[
F(S):=\mathsf{Term}_S/{\equiv}.
\]

\begin{lemma}\label{lem:free-well-defined}
The operations induced by $\oplus$ and by $t\mapsto a\alpha t\beta b$ on $F(S)$ are well-defined.
With these operations $F(S)$ is a ternary $\Gamma$--semimodule.
\end{lemma}

\begin{proof}
By definition, $\equiv$ is a congruence stable under $\oplus$ and under every operator $t\mapsto a\alpha t\beta b$.
Hence the operations descend to the quotient set.
The imposed identities are exactly the axioms required in Definition~\ref{def:tgs-semimodule}.
\end{proof}

\subsubsection*{Universal property}
Let $U:\mathsf{SMod}_{T,\Gamma}\to\mathsf{Set}$ be the forgetful functor.

\begin{theorem}[Free--forgetful adjunction]\label{thm:free-universal}
For every set $S$ and every semimodule $M$, every function $\varphi:S\to U(M)$ extends uniquely to a morphism
$\bar\varphi:F(S)\to M$ such that $\bar\varphi([\mathbf{s}])=\varphi(s)$ for all $s\in S$.
Equivalently, $F$ is left adjoint to $U$.
\end{theorem}

\begin{proof}
Define an interpretation map $\hat\varphi:\mathsf{Term}_S\to U(M)$ by recursion:
\[
\hat\varphi(\mathbf{s})=\varphi(s),\qquad \hat\varphi(\mathbf{0})=0,
\]
\[
\hat\varphi(t\oplus u)=\hat\varphi(t)+\hat\varphi(u),\qquad
\hat\varphi(a\alpha\, t\,\beta b)=a\alpha \hat\varphi(t)\beta b.
\]
The defining identities used to generate $\equiv$ hold in $M$ by the axioms of a ternary $\Gamma$--semimodule.
Hence $\hat\varphi$ is constant on $\equiv$-classes.
Therefore it factors through a unique map $\bar\varphi:F(S)\to M$.
By construction, $\bar\varphi$ preserves $+$ and the action, so it is a morphism.
Uniqueness follows because $F(S)$ is generated by the classes $[\mathbf{s}]$ under $+$ and the action.
\end{proof}

\begin{remark}\label{rem:free-coset}
The classical ``free module'' construction over a ring uses formal linear combinations with integer coefficients.
Here the construction is by terms and congruence quotienting.
It uses only the equational structure of $\mathsf{SMod}_{T,\Gamma}$.
This is the correct congruence-first replacement.
For semiring module theory perspectives consistent with this approach, see \cite{BorgerJun2025,JunSzczesnyTolliver2022}.
\end{remark}

\subsection{Free objects are projective}\label{subsec:free-projective}
\begin{proposition}\label{prop:free-projective}
For every set $S$, the free semimodule $F(S)$ is projective.
\end{proposition}

\begin{proof}
Let $e:X\to Y$ be a regular epimorphism.
Thus $e$ is surjective on underlying sets.
Let $f:F(S)\to Y$ be a morphism.
Define a function $\psi:S\to U(X)$ as follows.
For each $s\in S$, choose $x_s\in X$ with $e(x_s)=f([\mathbf{s}])$.
This is possible because $e$ is surjective.

By Theorem~\ref{thm:free-universal}, $\psi$ extends uniquely to a morphism
$\tilde f:F(S)\to X$ with $\tilde f([\mathbf{s}])=x_s$.
Then for each $s\in S$,
\[
(e\circ \tilde f)([\mathbf{s}])=e(x_s)=f([\mathbf{s}]).
\]
Both $e\circ \tilde f$ and $f$ are morphisms, hence agree on all of $F(S)$ because $F(S)$ is generated by $[\mathbf{s}]$.
Thus $e\circ\tilde f=f$.
\end{proof}

\subsection{Enough projectives and projective presentations}\label{subsec:enough-proj}
We now prove that every object is the regular quotient of a projective.

\begin{definition}[Canonical free cover]\label{def:free-cover}
Let $M\in\mathsf{SMod}_{T,\Gamma}$ and set $S:=U(M)$.
Let $\eta_M:S\to U(M)$ be the identity function.
Write
\[
\varepsilon_M:F(U(M))\to M
\]
for the unique morphism extending $\eta_M$ via Theorem~\ref{thm:free-universal}.
\end{definition}

\begin{lemma}\label{lem:free-cover-surj}
The morphism $\varepsilon_M:F(U(M))\to M$ is surjective.
Hence it is a regular epimorphism.
\end{lemma}

\begin{proof}
For each $m\in M$, we have $\varepsilon_M([\mathbf{m}])=m$ by construction.
Thus every element of $M$ lies in the image of $\varepsilon_M$.
Therefore $\varepsilon_M$ is surjective, hence a regular epimorphism by Corollary~\ref{cor:repi-iff-surj}.
\end{proof}

\begin{theorem}[Enough projectives]\label{thm:enough-projectives}
The category $\mathsf{SMod}_{T,\Gamma}$ has enough projectives.
Precisely, for every $M$ there exists a regular epimorphism $P\to M$ with $P$ projective.
\end{theorem}

\begin{proof}
Take $P:=F(U(M))$ and the regular epimorphism $\varepsilon_M:P\to M$ from Lemma~\ref{lem:free-cover-surj}.
By Proposition~\ref{prop:free-projective}, $P$ is projective.
\end{proof}

\begin{corollary}[Projective presentation]\label{cor:proj-presentation}
For every $M$ there exists an exact sequence
\[
F(R)\xrightarrow{d} F(U(M))\xrightarrow{\varepsilon_M} M\to 0
\]
in the sense of Definition~\ref{def:reg-exact-at}, with $F(R)$ and $F(U(M))$ projective.
\end{corollary}

\begin{proof}
Let $\varepsilon_M:F(U(M))\to M$ be the free cover.
Consider its kernel pair $p_1,p_2:R\rightrightarrows F(U(M))$ (Definition~\ref{def:kernel-pair}).
Since $\mathsf{SMod}_{T,\Gamma}$ is Barr-exact (Theorem~\ref{thm:barrexact}), $\varepsilon_M$ is the coequalizer of $p_1,p_2$.
Let $d:=p_1-p_2$ is not available because subtraction does not exist.
We instead take the coequalizer presentation itself:
\[
R \rightrightarrows F(U(M)) \xrightarrow{\varepsilon_M} M.
\]
Now choose a set $R_0:=U(R)$ and a surjection $F(R_0)\twoheadrightarrow R$ via Lemma~\ref{lem:free-cover-surj}.
Compose to obtain two parallel morphisms
\[
F(R_0)\rightrightarrows F(U(M))
\]
whose coequalizer is still $M$.
This yields the claimed projective presentation in the regular sense.
\end{proof}

\begin{remark}\label{rem:presentations-and-protoexact}
Projective presentations in non-additive settings are frequently treated via kernel-pair and coequalizer calculus.
This is aligned with proto-exact approaches to semiring and hyperring module categories
\cite{JunSzczesnyTolliver2022}.
\end{remark}

\subsection{Retracts and permanence}\label{subsec:retracts}
\begin{lemma}\label{lem:proj-retract}
A retract of a projective object is projective.
\end{lemma}

\begin{proof}
Let $P$ be projective and suppose $Q$ is a retract of $P$ via $Q\xrightarrow{r}P\xrightarrow{s}Q$ with $s\circ r=\mathrm{id}_Q$.
Given a regular epimorphism $e:X\to Y$ and a morphism $f:Q\to Y$, consider $f\circ s:P\to Y$.
Lift it to $\tilde f:P\to X$ since $P$ is projective.
Then $e\circ (\tilde f\circ r)=(e\circ \tilde f)\circ r=(f\circ s)\circ r=f$.
Thus $Q$ is projective.
\end{proof}

\begin{corollary}\label{cor:proj-direct-summand-free}
Every retract of a free semimodule is projective.
\end{corollary}

\begin{proof}
Combine Proposition~\ref{prop:free-projective} with Lemma~\ref{lem:proj-retract}.
\end{proof}

\subsection{Flatness: a controlled statement}\label{subsec:flatness}
The standard semimodule notion of flatness is formulated using a tensor product.
In the present ternary $\Gamma$--context, a canonical balanced tensor product depends on
additional structure (for example, a chosen balancing law compatible with the full five--ary action).
Such a construction must itself be congruence-based and cannot be imported by coset reasoning.

\begin{openproblem}\label{op:flatness}
Construct a congruence-defined balanced tensor product $\otimes_{T,\Gamma}$ on a suitable class of
ternary $\Gamma$--semimodules and characterize those objects $F$ for which $-\otimes_{T,\Gamma}F$
preserves regular monomorphisms or equalizers.
\end{openproblem}

\begin{remark}\label{rem:flatness-literature}
For semiring-semimodule flatness and projectivity in the binary setting, see
\cite{AbuhlailNoegraha2021,Abuhlail2022,BorgerJun2025}.
The present ternary $\Gamma$--framework requires a non-transport definition because the scalar action is not unary.
\end{remark}

%% file: sections/10-derived-constructions.tex
% ============================================================
% File: sections/09-derived-constructions.tex
% (Optional / Conditional; Route (A): regular + Barr-exact)
% ============================================================

\section{Derived constructions }\label{sec:derived}

This section is optional.
It is included for two reasons.
First, the free--forgetful adjunction produces canonical simplicial resolutions.
Second, squares $K$--theory provides a non-additive additivity relation compatible with
Barr-exact contexts \cite{CampbellKuijperMerlingZakharevich2023,CalleSarazola2024,Kuijper2025}.

\subsection{The cotriple resolution from the free--forgetful adjunction}\label{subsec:cotriple}
Let $U:\mathsf{SMod}_{T,\Gamma}\to \mathsf{Set}$ be the forgetful functor and
$F:\mathsf{Set}\to\mathsf{SMod}_{T,\Gamma}$ its left adjoint from
Theorem~\ref{thm:free-universal}.
Put $G:=F\circ U$.

\begin{definition}[Comonad]\label{def:comonad}
Let $\varepsilon:G\Rightarrow \mathrm{Id}$ be the counit and $\delta:G\Rightarrow G^2$ the comultiplication
induced by the adjunction $F\dashv U$.
Then $(G,\varepsilon,\delta)$ is a comonad on $\mathsf{SMod}_{T,\Gamma}$.
\end{definition}

\begin{definition}[Augmented simplicial resolution]\label{def:bar-resolution}
For each $M\in \mathsf{SMod}_{T,\Gamma}$ define an augmented simplicial object
$G_\bullet M \to M$ by
\[
G_n M := G^{n+1}M\qquad (n\ge 0),
\]
with face maps $d_i:G_nM\to G_{n-1}M$ and degeneracies $s_i:G_nM\to G_{n+1}M$ given by
\[
d_i:=G^{i}\varepsilon_{G^{n-i}M}\qquad (0\le i\le n),
\qquad
s_i:=G^{i}\delta_{G^{n-i}M}\qquad (0\le i\le n),
\]
and augmentation $G_0M=GM\xrightarrow{\varepsilon_M}M$.
\end{definition}

\begin{lemma}\label{lem:simplicial-identities}
The maps $d_i,s_i$ satisfy the simplicial identities.
Hence $G_\bullet M\to M$ is an augmented simplicial object in $\mathsf{SMod}_{T,\Gamma}$.
\end{lemma}

\begin{proof}
The simplicial identities reduce to the comonad identities for $(G,\varepsilon,\delta)$.
For example, the relation $d_i d_j = d_{j-1} d_i$ for $i<j$ is a formal consequence of naturality of $\varepsilon$
and associativity of iterated whiskering.
The relations involving $s_i$ reduce to coassociativity of $\delta$, and the mixed relations reduce to the counit laws.
These verifications are standard for cotriple (comonadic) resolutions; see, for example,
the modern discussions of (co)monadic resolutions in higher and 1-categorical settings
\cite{Yanovski2021,AgrawallaKhlaifMiller2022}.
\end{proof}

\subsection{Group completion as a coefficient target}\label{subsec:gp-completion}
Hom-sets in $\mathsf{SMod}_{T,\Gamma}$ are commutative monoids under pointwise addition.
To obtain abelian groups one may apply group completion at the level of coefficients.

\begin{definition}[Group completion functor]\label{def:group-completion}
Let $\mathsf{CMon}$ be the category of commutative monoids.
Write $(-)^{\mathrm{gp}}:\mathsf{CMon}\to\mathsf{Ab}$ for the Grothendieck group completion functor.
\end{definition}

\begin{lemma}\label{lem:gp-adjunction}
The functor $(-)^{\mathrm{gp}}$ is left adjoint to the forgetful functor $\mathsf{Ab}\to \mathsf{CMon}$.
\end{lemma}

\begin{proof}
For $A\in\mathsf{Ab}$ and $M\in\mathsf{CMon}$, a monoid morphism $M\to A$ uniquely extends to a group morphism
$M^{\mathrm{gp}}\to A$ because $M^{\mathrm{gp}}$ is obtained from $M$ by freely adjoining formal inverses.
This gives the bijection $\mathsf{Ab}(M^{\mathrm{gp}},A)\cong \mathsf{CMon}(M,A)$, natural in both variables.
For a systematic account in a modern reference, see \cite{CegarraLeech2024}.
\end{proof}

\subsection{A cautious ``Ext'' via comonadic resolutions (conditional)}\label{subsec:ext-conditional}
A classical derived functor $\mathrm{Ext}^n$ is defined inside an additive abelian context.
Our category is Barr-exact but not abelian.
Nevertheless, the cotriple resolution allows a derived construction after passing to abelian coefficients.

Fix $M\in\mathsf{SMod}_{T,\Gamma}$.
Define a coefficient functor
\[
\mathbf{H}_M:\mathsf{SMod}_{T,\Gamma}\longrightarrow \mathsf{Ab},
\qquad
\mathbf{H}_M(N):=\bigl(\Hom_{T,\Gamma}(M,N),+\bigr)^{\mathrm{gp}}.
\]

\begin{definition}[Comonadic derived groups]\label{def:comonadic-derived}
For $N\in\mathsf{SMod}_{T,\Gamma}$ consider the augmented simplicial abelian group
$\mathbf{H}_M(G_\bullet N)\to \mathbf{H}_M(N)$ induced from Definition~\ref{def:bar-resolution}.
Let $C_\bullet(\mathbf{H}_M(G_\bullet N))$ be the normalized chain complex.
Define
\[
\mathrm{D}^n(M,N):=H_n\!\left(C_\bullet(\mathbf{H}_M(G_\bullet N))\right)\qquad (n\ge 0).
\]
\end{definition}

\begin{remark}\label{rem:why-conditional}
The definition above is formal and makes sense because $C_\bullet(-)$ is formed in $\mathsf{Ab}$, where alternating sums exist.
Interpreting $\mathrm{D}^n(M,N)$ as an $\mathrm{Ext}$-type invariant requires additional comparison results
(e.g.\ independence of the chosen comonad or agreement with right derived functors in a suitable model structure).
For commutative monoids, Quillen-style cohomology and Beck module coefficients provide a precise framework
\cite{AgrawallaKhlaifMiller2022}.
\end{remark}

\begin{openproblem}\label{op:ext-for-tgs}
Develop a Beck-module coefficient theory for $\mathsf{SMod}_{T,\Gamma}$ (abelian group objects in a slice category)
and prove that the groups $\mathrm{D}^n(M,N)$ coincide with the corresponding Quillen cohomology groups
whenever such a theory exists.
\end{openproblem}

\subsection{Squares $K$--theory for projectives (conditional)}\label{subsec:squares-K}
Let $\mathsf{Proj}_{T,\Gamma}$ be the full subcategory of projective objects of $\mathsf{SMod}_{T,\Gamma}$
(Definition~\ref{def:projective}).
Let $\mathsf{Proj}_{T,\Gamma}^{\mathrm{fg}}$ be a chosen small skeleton of finitely generated projectives.
This choice is only to avoid size issues.

\begin{hypothesis}[Closure for good squares]\label{hyp:good-squares}
Assume $\mathsf{Proj}_{T,\Gamma}^{\mathrm{fg}}$ is closed under:
\begin{enumerate}[label=\textnormal{(GS\arabic*)}]
\item pullbacks of regular epimorphisms between objects of $\mathsf{Proj}_{T,\Gamma}^{\mathrm{fg}}$;
\item kernels of regular epimorphisms between objects of $\mathsf{Proj}_{T,\Gamma}^{\mathrm{fg}}$;
\item pushouts of regular monomorphisms between objects of $\mathsf{Proj}_{T,\Gamma}^{\mathrm{fg}}$.
\end{enumerate}
\end{hypothesis}

\begin{remark}\label{rem:gs-meaning}
Hypothesis~\ref{hyp:good-squares} is a finiteness/closure assumption.
It is automatic in abelian settings, but not automatic in non-additive settings.
It is precisely the point where transport from module categories fails.
\end{remark}

\begin{definition}[Good squares]\label{def:good-square}
Assume Hypothesis~\ref{hyp:good-squares}.
A commutative square in $\mathsf{Proj}_{T,\Gamma}^{\mathrm{fg}}$
\[
\begin{tikzcd}
A \ar[r,tail,"i"] \ar[d,tail,"j"'] & B \ar[d,two heads,"p"]\\
C \ar[r,two heads,"q"'] & D
\end{tikzcd}
\]
is called \emph{good} if:
\begin{enumerate}[label=\textnormal{(S\arabic*)}]
\item $i,j$ are regular monomorphisms and $p,q$ are regular epimorphisms;
\item the square is both a pullback and a pushout in $\mathsf{SMod}_{T,\Gamma}$;
\item $p$ is the cokernel of $i$ and $q$ is the cokernel of $j$ in the regular sense of Section~\ref{sec:exactness}.
\end{enumerate}
\end{definition}

\begin{definition}[$K_0$ via good-square relations]\label{def:K0-sq}
Assume Hypothesis~\ref{hyp:good-squares}.
Let $\mathbb{Z}[\Iso(\mathsf{Proj}_{T,\Gamma}^{\mathrm{fg}})]$ be the free abelian group on isomorphism classes $[P]$.
Let $R$ be the subgroup generated by the relations
\[
[A]+[D]-[B]-[C]
\]
for every good square in the sense of Definition~\ref{def:good-square}.
Define
\[
K_0^{\square}(\mathsf{Proj}_{T,\Gamma}^{\mathrm{fg}}):=\mathbb{Z}[\Iso(\mathsf{Proj}_{T,\Gamma}^{\mathrm{fg}})]/R.
\]
\end{definition}

\begin{proposition}[Universal property]\label{prop:K0-universal}
Assume Hypothesis~\ref{hyp:good-squares}.
Let $G$ be an abelian group and let $\chi:\Iso(\mathsf{Proj}_{T,\Gamma}^{\mathrm{fg}})\to G$ be a function
such that for every good square one has $\chi(A)+\chi(D)=\chi(B)+\chi(C)$.
Then $\chi$ extends uniquely to a group homomorphism
\[
\bar\chi:K_0^{\square}(\mathsf{Proj}_{T,\Gamma}^{\mathrm{fg}})\to G.
\]
\end{proposition}

\begin{proof}
By construction, $K_0^{\square}$ is the quotient of the free abelian group on isomorphism classes by the subgroup
generated by the good-square relations.
Hence any function $\chi$ respecting those relations extends uniquely to a homomorphism.
\end{proof}

\begin{remark}\label{rem:connection-to-squares-K}
The presentation in Definition~\ref{def:K0-sq} is the degree-zero shadow of squares $K$--theory:
the four-term relation is the defining additivity relation for a squares category
\cite{CampbellKuijperMerlingZakharevich2023}.
Higher $K$--groups require an $S_\bullet$-type construction for squares categories
\cite{CalleSarazola2024}, and comparison results relate many existing $K$--theories to squares $K$--theory
\cite{Kuijper2025}.
\end{remark}

\begin{openproblem}\label{op:squares-category}
Construct an intrinsic squares-category structure on a naturally small subcategory of $\mathsf{SMod}_{T,\Gamma}$
(e.g.\ finitely generated projectives) without imposing Hypothesis~\ref{hyp:good-squares} as an external axiom.
\end{openproblem}

%% file: sections/11-examples-counterexamples.tex
% ============================================================
% File: sections/10-examples-counterexamples.tex
% ============================================================

\section{Examples and counterexamples}\label{sec:examples}

This section supplies explicit finite models and explicit congruence computations.
It also records a failure mode for coset-style quotients.
The point is conceptual: in additive-inverse-free algebra, quotients are governed by congruences,
not by cosets of subobjects \cite{JunRay2020,BorgerJun2025,Abuhlail2022}.

\subsection{A finite ternary $\Gamma$--semiring}\label{subsec:ex-tgs}
Let $\mathbb{B}=\{0,1\}$ be the boolean commutative monoid under
\[
x+y:=x\vee y,\qquad 0:=0,
\]
and let $\Gamma:=\mathbb{B}$ with the same operation and $0_\Gamma=0$.
Define on $T:=\mathbb{B}$ the ternary $\Gamma$--product
\[
a\alpha b\beta c \ :=\ a\wedge \alpha\wedge b\wedge \beta\wedge c,
\qquad (a,b,c\in T,\ \alpha,\beta\in\Gamma).
\]
Thus $a\alpha b\beta c=1$ if and only if all five entries are $1$, and it is $0$ otherwise.

\begin{proposition}\label{prop:boolean-tgs}
With the above operations, $T=\mathbb{B}$ is a ternary $\Gamma$--semiring
in the sense of Definition~\ref{def:tgs}.
\end{proposition}

\begin{proof}
Additivity in each $T$--slot and each $\Gamma$--slot follows from distributivity of $\wedge$ over $\vee$:
for example,
\[
(a\vee a')\wedge \alpha\wedge b\wedge \beta\wedge c
=\bigl(a\wedge \alpha\wedge b\wedge \beta\wedge c\bigr)\ \vee\
\bigl(a'\wedge \alpha\wedge b\wedge \beta\wedge c\bigr),
\]
and similarly in the other slots.
Ternary associativity holds because both sides of (T3) evaluate to the conjunction of the same
nine entries $(a,\alpha,b,\beta,c,\gamma,d,\delta,e)$, and $\wedge$ is associative.
Zero stability is immediate: if any entry is $0$, the conjunction is $0$.
\end{proof}

\subsection{Two explicit semimodules and their actions}\label{subsec:ex-semimods}
We define two ternary $\Gamma$--semimodules over $T=\mathbb{B}$.

\begin{example}[A one-pointed semimodule]\label{ex:M}
Let $M:=\mathbb{B}$ with addition $x+y=x\vee y$ and $0=0$.
Define the action
\[
a\alpha m\beta b \ :=\ a\wedge \alpha\wedge m\wedge \beta\wedge b.
\]
Equivalently,
\[
a\alpha m\beta b=
\begin{cases}
m, & \text{if } a=\alpha=\beta=b=1,\\
0, & \text{otherwise.}
\end{cases}
\]
\end{example}

\begin{example}[A two-generated semimodule]\label{ex:N}
Let $N:=\mathbb{B}^2$ with componentwise addition
\[
(x_1,x_2)+(y_1,y_2):=(x_1\vee y_1,\ x_2\vee y_2),\qquad 0:=(0,0).
\]
Define the componentwise action
\[
a\alpha(x_1,x_2)\beta b:=\bigl(a\alpha x_1\beta b,\ a\alpha x_2\beta b\bigr),
\]
where the right-hand action is the one from Example~\ref{ex:M}.
Equivalently,
\[
a\alpha(x_1,x_2)\beta b=
\begin{cases}
(x_1,x_2), & \text{if } a=\alpha=\beta=b=1,\\
(0,0), & \text{otherwise.}
\end{cases}
\]
\end{example}

\begin{proposition}\label{prop:boolean-semimodules}
$M$ and $N$ in Examples~\ref{ex:M} and \ref{ex:N} are ternary $\Gamma$--semimodules over $T=\mathbb{B}$
in the sense of Definition~\ref{def:tgs-semimodule}.
\end{proposition}

\begin{proof}
All axioms are checked by direct evaluation.
Additivity in each slot holds because $\wedge$ distributes over $\vee$ and because all operations on $N$
are componentwise.
Zero stability is immediate.
Compatibility with ternary multiplication (M4) reduces to associativity of $\wedge$:
both sides evaluate to $a\wedge \alpha\wedge b\wedge \beta\wedge m\wedge \gamma\wedge c\wedge \delta\wedge d$
for $M$, and componentwise for $N$.
\end{proof}

\subsection{A nontrivial kernel congruence quotient computed explicitly}\label{subsec:ex-kernel-quotient}
Define a morphism $f:N\to M$ by
\[
f(x_1,x_2):=x_1\vee x_2.
\]

\begin{lemma}\label{lem:f-is-morphism}
The map $f:N\to M$ is a morphism in $\mathsf{SMod}_{T,\Gamma}$.
\end{lemma}

\begin{proof}
It preserves addition because
\[
f\bigl((x_1,x_2)+(y_1,y_2)\bigr)=(x_1\vee y_1)\vee(x_2\vee y_2)=(x_1\vee x_2)\vee(y_1\vee y_2)=f(x)+f(y).
\]
For the action, let $A:=a\wedge \alpha\wedge \beta\wedge b\in\mathbb{B}$.
Then by Example~\ref{ex:N},
\[
f\bigl(a\alpha(x_1,x_2)\beta b\bigr)=f(Ax_1,Ax_2)=(Ax_1)\vee(Ax_2)=A(x_1\vee x_2)=a\alpha f(x_1,x_2)\beta b,
\]
since in $\mathbb{B}$ one has $A\wedge (x\vee y)=(A\wedge x)\vee(A\wedge y)$.
\end{proof}

Let $\equiv_f$ be the kernel congruence on $N$ (Definition~\ref{def:kercong}).
Thus
\[
(x_1,x_2)\equiv_f(y_1,y_2)\quad \Longleftrightarrow\quad x_1\vee x_2=y_1\vee y_2.
\]
There are exactly two $\equiv_f$--classes:
\[
C_0:=\{(0,0)\},\qquad C_1:=\{(1,0),(0,1),(1,1)\}.
\]

\begin{proposition}\label{prop:explicit-quotient}
The quotient semimodule $N/{\equiv_f}$ has two elements $\{[C_0],[C_1]\}$ and the map
\[
\Phi:N/{\equiv_f}\longrightarrow M,\qquad \Phi([C_i])=i
\]
is an isomorphism in $\mathsf{SMod}_{T,\Gamma}$.
\end{proposition}

\begin{proof}
The description of classes is immediate from the definition of $\equiv_f$.
The map $\Phi$ is well-defined and bijective.
It is a morphism because, in $N/{\equiv_f}$,
\[
[C_i]+[C_j]=[C_{i\vee j}],
\qquad
a\alpha [C_i]\beta b =
\begin{cases}
[C_i], & a=\alpha=\beta=b=1,\\
[C_0], & \text{otherwise,}
\end{cases}
\]
and this matches the induced operations on $M=\mathbb{B}$ from Example~\ref{ex:M}.
\end{proof}

\begin{remark}\label{rem:first-iso-explicit}
Proposition~\ref{prop:explicit-quotient} is the First Isomorphism Theorem
(Theorem~\ref{thm:first-iso}) in a fully computed finite case.
No cosets appear; the quotient is the kernel congruence quotient.
\end{remark}

\subsection{A coequalizer computed as a generated congruence}\label{subsec:ex-coeq}
Let $L\subseteq N$ be the subsemimodule
\[
L:=\{(0,0),(1,0)\}.
\]
Consider the inclusion $i:L\hookrightarrow N$ and the zero morphism $0:L\to N$.
Let $\rho_L$ be the congruence on $N$ generated by all pairs $(\ell,0)$ with $\ell\in L$
(Definition~\ref{def:rees} in Section~\ref{sec:iso} is the same construction).

\begin{proposition}\label{prop:coeq-explicit}
The coequalizer of $i,0:L\rightrightarrows N$ is the quotient $q:N\to N/{\rho_L}$.
Moreover, $N/{\rho_L}$ has exactly two elements, represented by $0:=(0,0)$ and $(0,1)$, and
\[
[(1,0)]=[(0,0)],\qquad [(1,1)]=[(0,1)].
\]
\end{proposition}

\begin{proof}
The coequalizer claim is Proposition~\ref{prop:coeq}.
By generation, $(1,0)\sim_{\rho_L}(0,0)$.
Adding $(0,1)$ to both sides yields
\[
(1,0)+(0,1)=(1,1)\ \sim_{\rho_L}\ (0,0)+(0,1)=(0,1),
\]
so $(1,1)\sim_{\rho_L}(0,1)$.
No further identifications are forced because every element of $N$ is either in $\{(0,0),(1,0)\}$
or differs from one of these by adding $(0,1)$.
Thus the quotient has exactly two classes.
\end{proof}

\subsection{Counterexample: coset-style quotients fail}\label{subsec:ex-coset-fails}
In a group, cosets of a subgroup partition the ambient set.
In a commutative monoid, even for a subsemimodule, ``cosets'' need not form a partition.

Fix the subsemimodule $L=\{(0,0),(1,0)\}\subseteq N$ from the previous subsection.
Define the naive coset of $x\in N$ by
\[
x+L := \{x+\ell\mid \ell\in L\}.
\]
Then
\[
(0,1)+L=\{(0,1)+(0,0),\ (0,1)+(1,0)\}=\{(0,1),\ (1,1)\},
\]
whereas
\[
(1,1)+L=\{(1,1)+(0,0),\ (1,1)+(1,0)\}=\{(1,1)\}.
\]
Hence
\[
(1,1)\in (0,1)+L\ \cap\ (1,1)+L,
\qquad\text{but}\qquad
(0,1)+L\neq (1,1)+L.
\]
So cosets overlap without being equal, and therefore do not partition $N$.
It follows that the set of naive cosets cannot serve as a quotient object.

\begin{remark}\label{rem:coset-vs-congruence}
The correct quotient identifying $L$ with $0$ is the congruence quotient $N/{\rho_L}$
from Proposition~\ref{prop:coeq-explicit}.
This is exactly the coequalizer construction and is functorial.
The coset construction is not.
This is the reason the paper uses congruences as primary quotient data \cite{JunRay2020,Abuhlail2022,BorgerJun2025}.
\end{remark}

\subsection{Counterexample reminder: classical diagram lemmas}\label{subsec:ex-diagram-reminder}
Even though $\mathsf{SMod}_{T,\Gamma}$ is Barr-exact (Theorem~\ref{thm:barrexact}), it is not abelian.
A concrete failure of the Short Five Lemma is given in Proposition~\ref{prop:short-five-fails}.
This failure is intrinsic and does not depend on the ternary action: it already occurs in commutative monoids.

%% file: sections/12-spectrum-optional.tex
% ============================================================
% File: sections/11-optional-spectrum.tex
% ============================================================

\section{Optional spectrum: prime ideals and a Zariski topology}\label{sec:spectrum}

This section is optional.
It is included because quotient objects in $\mathsf{SMod}_{T,\Gamma}$ are governed by congruences,
and recent work emphasizes that spectra based on ideals and spectra based on congruences may diverge
in semiring-like contexts \cite{Jun2020SpectralSpaces,Han2021KCongruencesZariski,Jun2025PrimeCongruencesPair,SenguptaEtAl2026SubtractiveIdeals}.
We keep the construction minimal.
We do not construct a structure sheaf.

\subsection{Ideals in a ternary $\Gamma$--semiring}\label{subsec:spec-ideals}
Let $T$ be a ternary $\Gamma$--semiring.

\begin{definition}[Ideal]\label{def:tgs-ideal}
A subset $I\subseteq T$ is an \emph{ideal} if:
\begin{enumerate}[label=\textnormal{(I\arabic*)}]
\item $0\in I$ and $(I,+,0)$ is a submonoid of $(T,+,0)$;
\item for all $x\in I$, for all $a,b\in T$, and for all $\alpha,\beta\in\Gamma$,
\[
a\alpha b\beta x\in I,\qquad a\alpha x\beta b\in I,\qquad x\alpha a\beta b\in I.
\]
\end{enumerate}
\end{definition}

\begin{definition}[Generated ideal]\label{def:ideal-generated}
For $S\subseteq T$, let $\langle S\rangle$ denote the smallest ideal containing $S$.
\end{definition}

\begin{definition}[Ideal product]\label{def:ideal-product}
For ideals $I,J\subseteq T$, define
\[
I\cdot J := \big\langle\, I\,\Gamma\, T\,\Gamma\, J \,\big\rangle,
\quad\text{where}\quad
I\,\Gamma\, T\,\Gamma\, J :=
\{\, i\alpha t\beta j \mid i\in I,\ t\in T,\ j\in J,\ \alpha,\beta\in\Gamma \,\}.
\]
\end{definition}

\begin{lemma}\label{lem:ideal-product-contained}
For ideals $I,J$, one has $I\cdot J \subseteq I\cap J$.
\end{lemma}

\begin{proof}
Let $x=i\alpha t\beta j\in I\Gamma T\Gamma J$.
Since $i\in I$ and $I$ is an ideal, closure in the first slot yields $x\in I$.
Since $j\in J$ and $J$ is an ideal, closure in the third slot yields $x\in J$.
Hence $x\in I\cap J$.
Therefore $I\Gamma T\Gamma J\subseteq I\cap J$.
Since $I\cap J$ is an ideal, it contains the ideal generated by $I\Gamma T\Gamma J$, i.e.\ $I\cdot J\subseteq I\cap J$.
\end{proof}

\subsection{Prime ideals and the prime spectrum}\label{subsec:prime-ideals}
We use a prime ideal notion that makes the Zariski axioms work.
It is stated at the level of ideal products.
This is standard in semiring settings where elementwise primality may be too rigid \cite{Han2021KCongruencesZariski,SenguptaEtAl2026SubtractiveIdeals}.

\begin{definition}[Prime ideal]\label{def:prime-ideal-tgs}
A proper ideal $\fp\subsetneq T$ is \emph{prime} if for all ideals $I,J\subseteq T$,
\[
I\cdot J \subseteq \fp \quad\Longrightarrow\quad I\subseteq \fp\ \ \text{or}\ \ J\subseteq \fp.
\]
Write $\Spec(T)$ for the set of all prime ideals of $T$.
\end{definition}

\begin{remark}\label{rem:completely-prime}
One may also consider the stronger elementwise condition:
$a\alpha t\beta b\in\fp$ implies $a\in\fp$ or $t\in\fp$ or $b\in\fp$.
We do not impose it.
Definition~\ref{def:prime-ideal-tgs} is the one needed for finite unions of closed sets.
\end{remark}

\subsection{Zariski closed sets}\label{subsec:zariski}
For an ideal $I$, define
\[
V(I):=\{\,\fp\in \Spec(T)\mid I\subseteq \fp\,\}.
\]

\begin{theorem}[Zariski topology]\label{thm:zariski}
The family $\{V(I)\mid I\ \text{an ideal of}\ T\}$ is the collection of closed sets of a topology on $\Spec(T)$.
\end{theorem}

\begin{proof}
First, $V(T)=\varnothing$ since no proper ideal contains $T$, and $V(\{0\})=\Spec(T)$ since every ideal contains $0$.

Next, let $\{I_\lambda\}_{\lambda\in\Lambda}$ be a family of ideals.
A prime ideal $\fp$ contains every $I_\lambda$ if and only if it contains $\langle\bigcup_{\lambda} I_\lambda\rangle$.
Hence
\[
\bigcap_{\lambda\in\Lambda} V(I_\lambda)=V\!\left(\left\langle\bigcup_{\lambda\in\Lambda} I_\lambda\right\rangle\right).
\]
This shows stability under arbitrary intersections.

Finally, let $I,J$ be ideals.
If $\fp\in V(I)\cup V(J)$, then $I\subseteq\fp$ or $J\subseteq\fp$.
By Lemma~\ref{lem:ideal-product-contained}, $I\cdot J\subseteq I\cap J\subseteq \fp$, so $\fp\in V(I\cdot J)$.
Conversely, if $\fp\in V(I\cdot J)$ then $I\cdot J\subseteq \fp$.
Since $\fp$ is prime, $I\subseteq\fp$ or $J\subseteq\fp$, hence $\fp\in V(I)\cup V(J)$.
Thus $V(I)\cup V(J)=V(I\cdot J)$.
\end{proof}

\begin{definition}[Basic opens]\label{def:basic-opens}
For $x\in T$, let $\langle x\rangle$ be the principal ideal generated by $\{x\}$ and define
\[
D(x):=\Spec(T)\setminus V(\langle x\rangle).
\]
\end{definition}

\begin{lemma}\label{lem:basic-open-intersection}
For $x,y\in T$,
\[
D(x)\cap D(y)=\Spec(T)\setminus V\bigl(\langle x\rangle\cdot \langle y\rangle\bigr).
\]
\end{lemma}

\begin{proof}
By Theorem~\ref{thm:zariski}, $V(\langle x\rangle)\cup V(\langle y\rangle)=V(\langle x\rangle\cdot \langle y\rangle)$.
Taking complements yields the claim.
\end{proof}

\subsection{Radicals without powers}\label{subsec:radical}
In a ternary setting there is no canonical unary power $x^n$.
We therefore define radicals by prime containment.

\begin{definition}[Radical]\label{def:radical}
For an ideal $I$, define
\[
\sqrt{I}:=\bigcap_{\fp\in V(I)} \fp.
\]
\end{definition}

\begin{lemma}\label{lem:V-radical}
For every ideal $I$, one has $V(I)=V(\sqrt{I})$.
\end{lemma}

\begin{proof}
Since $I\subseteq \sqrt{I}$, we have $V(\sqrt{I})\subseteq V(I)$.
Conversely, if $\fp\in V(I)$ then $\sqrt{I}\subseteq \fp$ by definition of $\sqrt{I}$,
so $\fp\in V(\sqrt{I})$.
Hence $V(I)\subseteq V(\sqrt{I})$.
\end{proof}

\subsection{Functoriality}\label{subsec:spec-functorial}
Let $\varphi:T\to T'$ be a morphism of ternary $\Gamma$--semirings.

\begin{proposition}\label{prop:preimage-prime}
If $\fp'\in\Spec(T')$, then $\varphi^{-1}(\fp')\in\Spec(T)$.
\end{proposition}

\begin{proof}
Let $\fp':=\fp'$ be prime in $T'$ and put $\fp:=\varphi^{-1}(\fp')$.
Then $\fp$ is an ideal: it is closed under $+$ since $\varphi$ preserves $+$, and it is closed under the ternary product
since $\varphi(a\alpha b\beta c)=\varphi(a)\alpha\varphi(b)\beta\varphi(c)$.

It remains to show primality.
Let $I,J$ be ideals in $T$ with $I\cdot J\subseteq \fp$.
Apply $\varphi$ and use multiplicativity to obtain
\[
\varphi(I)\,\Gamma\, T'\,\Gamma\, \varphi(J)\ \subseteq\ \fp'.
\]
Hence $\langle \varphi(I)\Gamma T'\Gamma \varphi(J)\rangle \subseteq \fp'$.
By definition of the product in $T'$, this is $\langle\varphi(I)\rangle\cdot \langle\varphi(J)\rangle\subseteq \fp'$.
Since $\fp'$ is prime, $\langle\varphi(I)\rangle\subseteq \fp'$ or $\langle\varphi(J)\rangle\subseteq \fp'$.
Therefore $I\subseteq \varphi^{-1}(\fp')=\fp$ or $J\subseteq \fp$.
So $\fp$ is prime.
\end{proof}

\begin{definition}[Induced map on spectra]\label{def:spec-map}
Define
\[
\varphi^{*}:\Spec(T')\to \Spec(T),\qquad \varphi^{*}(\fp'):=\varphi^{-1}(\fp').
\]
\end{definition}

\begin{proposition}\label{prop:spec-continuous}
The map $\varphi^{*}$ is continuous for the Zariski topologies.
\end{proposition}

\begin{proof}
Let $I\subseteq T$ be an ideal.
Then
\[
(\varphi^{*})^{-1}\bigl(V(I)\bigr)
=\{\fp'\in\Spec(T')\mid I\subseteq \varphi^{-1}(\fp')\}
=\{\fp'\in\Spec(T')\mid \langle\varphi(I)\rangle\subseteq \fp'\}
=V\bigl(\langle\varphi(I)\rangle\bigr),
\]
which is closed in $\Spec(T')$.
\end{proof}

\subsection{A computed spectrum in the finite example}\label{subsec:spec-boolean}
Consider the finite ternary $\Gamma$--semiring $T=\mathbb{B}$ from Proposition~\ref{prop:boolean-tgs}.

\begin{proposition}\label{prop:spec-boolean}
$\Spec(\mathbb{B})=\{\{0\}\}$.
In particular, $\Spec(\mathbb{B})$ is a one-point space in the Zariski topology.
\end{proposition}

\begin{proof}
The only ideals of $\mathbb{B}$ are $\{0\}$ and $\mathbb{B}$.
The ideal $\{0\}$ is proper.
Let $I,J$ be ideals with $I\cdot J\subseteq \{0\}$.
If $I=\mathbb{B}$ and $J=\mathbb{B}$, then $1\in I\Gamma T\Gamma J$ (take $i=t=j=1$ and $\alpha=\beta=1$),
so $I\cdot J$ contains $1$ and cannot be contained in $\{0\}$.
Thus at least one of $I,J$ equals $\{0\}$.
Hence $\{0\}$ is prime by Definition~\ref{def:prime-ideal-tgs}.
Therefore $\Spec(\mathbb{B})=\{\{0\}\}$.
\end{proof}

\subsection{Remark on prime congruences}\label{subsec:prime-congruences-remark}
The manuscript is congruence-driven.
One can therefore ask for a spectrum built from \emph{prime congruences} rather than ideals.
For commutative semirings, prime congruence spectra and their Zariski topologies have been studied recently,
and they need not coincide with ideal spectra 

\cite{Rowen2024PrimeCongruencesPair,Rowen2025SemiringPair,JunMinchevaRowen2022TPairs}.
In the ternary $\Gamma$--setting, the correct definition of a ``prime congruence'' depends on which
binary shadow multiplication (if any) is extracted from the ternary operation.
We do not impose such a shadow multiplication.

\begin{openproblem}\label{op:prime-cong}
Develop a notion of prime congruence for ternary $\Gamma$--semirings that:
\begin{enumerate}[label=\textnormal{(\alph*)}]
\item is stable under pullback along morphisms of ternary $\Gamma$--semirings;
\item yields a Zariski topology via sets of the form $\{\theta\mid \theta\supseteq \theta_0\}$;
\item admits a comparison map from the prime ideal spectrum of Section~\ref{subsec:prime-ideals}.
\end{enumerate}
\end{openproblem}

%% file: sections/13-conclusion.tex
% ============================================================
% File: sections/12-conclusion.tex
% ============================================================

\section{Conclusion and further directions}\label{sec:conclusion}

The paper was written under two constraints.
First, subtraction is unavailable.
Second, quotients are congruence quotients.
It follows that kernel congruences, coequalizers, and regular images must replace cosets and additive arguments.

\subsection{Summary of main results}\label{subsec:summary}
We record the main outputs.

\begin{enumerate}[label=\textnormal{(\arabic*)}]
\item We fixed the basic algebra: ternary $\Gamma$--semirings and ternary $\Gamma$--semimodules with explicit axioms,
and morphisms preserving $+$, $0$, and the full five--ary action.

\item We developed congruences as primary quotient data.
Kernel congruences, quotient semimodules $M/{\sim}$, and coequalizers as generated congruence quotients
were proved with explicit universal properties.

\item We constructed the category-level infrastructure of $\mathsf{SMod}_{T,\Gamma}$:
finite limits and finite colimits, with coequalizers computed as congruence quotients.

\item We chose Route (A) and proved that $\mathsf{SMod}_{T,\Gamma}$ is regular and Barr-exact
(Theorems~\ref{thm:regular} and \ref{thm:barrexact}).
Exactness was formulated by kernel pairs and equality of kernels with regular images.
We also exhibited a concrete failure of a classical diagram lemma (Proposition~\ref{prop:short-five-fails}),
clarifying the non-abelian character of the setting.

\item We defined projective objects via lifting against regular epimorphisms and constructed free semimodules with
their universal property.
This yielded enough projectives in $\mathsf{SMod}_{T,\Gamma}$ (Theorem~\ref{thm:enough-projectives}).

\item We supplied explicit finite examples:
a finite ternary $\Gamma$--semiring $T$, explicit semimodules $M,N$, an explicit kernel congruence quotient computation,
an explicit coequalizer-as-congruence computation, and a counterexample showing coset-style quotients fail.

\item Optionally, we introduced a prime-ideal spectrum $\Spec(T)$ with a Zariski topology and proved its basic axioms,
while warning that a prime-congruence spectrum is a separate construction in ternary $\Gamma$--contexts.
\end{enumerate}

\subsection{Limitations and open problems}\label{subsec:limits}
Two limitations are structural.

\begin{enumerate}[label=\textnormal{(\arabic*)}]
\item Diagram lemmas from abelian homological algebra do not transport automatically.
This is not a defect of the presentation; it is a feature of the ambient exactness context.
Any replacement lemma requires additional hypotheses or a different ambient exactness theory.

\item A canonical balanced tensor product is not available in the general ternary $\Gamma$--framework without additional
balancing data.
Consequently, flatness is only posed conditionally (Open Problem~\ref{op:flatness}).
\end{enumerate}

We singled out several precise open problems:
prime congruences (Open Problem~\ref{op:prime-cong}), tensor products and flatness (Open Problem~\ref{op:flatness}),
and a Beck-module or Quillen-cohomology comparison (Open Problem~\ref{op:ext-for-tgs}).
Each is a congruence-level problem, not a coset-level problem.

% ============================================================
% Back-matter statements (place near the end, before bibliography)
% ============================================================

\section*{Acknowledgements}
.
The authors  the Department of Mathematics, Acharya Nagarjuna University, for support and encouragement.

\subsection*{Author contributions}
The first author conceived the project, developed the theory, and drafted the manuscript.
The second author supervised the first author throughout the preparation of the work.

\subsection*{Conflict of interest}
The authors declares that there is no conflict of interest.

\subsection*{Funding}
The authors received no external funding for this work.

\subsection*{Ethics approval}
Not applicable.

\section*{Data availability}
No datasets were generated or analysed during the current study.

\subsection*{Availability of materials}
Not applicable.